\documentclass[reqno,11pt]{amsart}
\usepackage{amssymb,amsmath,amscd, amsthm}
\usepackage{cases} 
\usepackage{graphicx}
\usepackage{wrapfig}
\usepackage{color}
 \usepackage{bm}
 
\usepackage{newtxmath,newtxtext}

\makeatletter
\@addtoreset{equation}{section}
\def\theequation{\thesection.\arabic{equation}}
\makeatother
  \makeatletter
  \def\@listi{%
  \leftmargin = 0pt \rightmargin = 0pt
     \labelwidth\leftmargin \advance\labelwidth-\labelsep
     \partopsep  = 0pt \itemsep       = 0pt
     \itemindent = 0pt \listparindent = 0pt}
  \let\@listI\@listi
  \@listi
  \def\@listii{%
     \leftmargin = 0pt \rightmargin = 0pt
     \labelwidth\leftmargin \advance\labelwidth-\labelsep
     \topsep     = 0pt \partopsep     = 0pt \itemsep   = 0pt
     \itemindent = 0pt \listparindent = 0pt}
  \let\@listiii\@listii
  \let\@listiv\@listii
  \let\@listv\@listii
  \let\@listvi\@listii
  \makeatother
  
\theoremstyle{plain}
\newtheorem{theorem}{Theorem}[section]
\newtheorem*{theorem*}{Theorem}
\newtheorem{proposition}[theorem]{Proposition}            
\newtheorem{corollary}[theorem]{Corollary}                
\newtheorem{lemma}[theorem]{Lemma}

\newtheorem{propo}{Proposition}
\newtheorem{theor}{Theorem}

\newtheorem{corol}[theor]{Corollary}

\theoremstyle{definition}\newtheorem{remark}[theorem]{Remark}

\newtheorem{definition}[theorem]{Definition}

\DeclareMathOperator{\tr}{tr}

\begin{document}
\title[Zero mean curvature surfaces of Riemann type]{Causal characters of zero mean curvature surfaces of Riemann type in the Lorentz-Minkowski $3$-space}
\author[S.\ Akamine]{Shintaro Akamine$^*$}
\address{Graduate School of mathematics \endgraf 
Kyushu University \endgraf
744 Motooka, Nishi-ku, \endgraf
Fukuoka 819-0395, Japan
}

\subjclass[2010]{ 
Primary 53A10; Secondary 53A35, 53C50.}

\keywords{ 
Causal character; zero mean curvature surface of Riemann type.
}
\thanks{ 
$^*$Research Fellow of the Japan Society for the Promotion of Science.
}
\email{s-akamine@math.kyushu-u.ac.jp}
\begin{abstract}
A zero mean curvature surface in the Lorentz-Minkowski $3$-space is said to be of Riemann type if it is foliated by circles and at most countably many straight lines in parallel planes. We classify all zero mean curvature surfaces of Riemann type according to their causal characters, and as a corollary, we prove that if a zero mean curvature surface of Riemann type has exactly two causal characters, then the lightlike part of the surface is a part of a straight line.
\end{abstract}
\maketitle
\section{Introduction\label{intro}}
We denote by $\mathbb{L}^3=\{(x,y,t)\mid x,y,t \in \mathbb{R}\}$ the Lorentz-Minkowski $3$-space with the metric $\langle \ ,\  \rangle$ of signature $(+,+,-)$. A zero mean curvature (ZMC, for short) surface in $\mathbb{L}^3$ is a surface whose mean curvature is identically zero. A spacelike one and timelike one are often said to be a maximal surface and a timelike minimal surface, respectively. There has been a lot of studies about these classes. For example, there are Weierstrass-Enneper type representation formulas for maximal surfaces \cite{Kobayashi} and timelike minimal surfaces \cite{Konderak,Magid} similar to the case of minimal surfaces in the Euclidean 3-space $\mathbb{E}^3$. On the other hand, in contrast to the case of minimal surfaces in $\mathbb{E}^3$, motivated by Calabi's work \cite{Calabi}, Cheng and Yau \cite{CY} showed that the only complete maximal surfaces are spacelike planes. 
From this fact, more general classes of maximal surfaces with singularities such as {\it maxfaces} in \cite{UY} or {\it generalized maximal surfaces} in \cite{Estudillo} have been frequently studied. Similarly, timelike minimal surfaces with singularities have been studied and many examples constructed. Note that we use the word singularities in the sense of points where a surface is not an immersion. Moreover, these surfaces often can be real analytically extended to ZMC surfaces which have more than one causal character.  
For example, motivated by the works \cite{Gu1, Gu2, Gu3,KKSY,Klyachin}, in \cite{FujimoriETAL3},  Fujimori, Kim, Koh, Rossman, Shin, Umehara, Yamada, and Yang showed that any maxface which has non-degenerate fold-singularities \cite[Definition 2.13]{FujimoriETAL3} is extended real analytically to a timelike minimal surface. In this extension, the set of fold singularities of the maxface forms a non-degenerate null curve on a ZMC surface, where a non-degenerate null curve is a regular curve whose velocity vector field is lightlike and not proportional to its acceleration vector field everywhere. On the other hand, there are no similar results for ZMC surfaces whose lightlike parts are straight lines, which degenerate everywhere. 
However, several ZMC surfaces which have exactly two causal characters and their lightlike parts are straight lines were constructed in \cite{FujimoriETAL1}, and there is a one-parameter family of ZMC surfaces which change their causal characters from spacelike to timelike across a lightlike line constructed in \cite{FujimoriETAL2}. These surfaces are the only known examples of ZMC surfaces which change their causal characters from spacelike to timelike across a lightlike line. From this viewpoint, the study of ZMC surfaces which have more than one causal character and especially the construction of examples of ZMC surfaces containing lightlike lines are important.

In this paper, we deal with non-rotational ZMC surfaces foliated by circles and at most countably many straight lines in parallel planes, called ZMC surfaces of Riemann type, and study the causal characters of these surfaces. The notion of circles in $\mathbb{L}^3$ is defined by L\'opez (see, for example, \cite{LLR,Lopez1}), and circles in $\mathbb{L}^3$ are classified into three cases: Euclidean circles in spacelike planes, hyperbolas in timelike planes with lightlike asymptotes, and parabolas in lightlike planes with lightlike axes according to causal characters of planes containing circles (see also Remark \ref{remark:parallel planes}). Therefore, the class of ZMC surfaces of Riemann type foliated by these circles is large enough to include some important examples of ZMC surfaces containing lightlike lines.

Spacelike and timelike ZMC surfaces of Riemann type and, more generally, constant mean curvature surfaces foliated by circles in $\mathbb{L}^3$ have been studied in \cite{LLR,Lopez1,Lopez2}, etc. L\'opez \cite{Lopez1} and Honda-Koiso-Tanaka \cite{HKT} constructed the following parametrizations of ZMC surfaces of Riemann type foliated by circles in spacelike parallel planes which appear as solutions with non-constant radii of a system of ordinary differential equations or ODEs, for short (see Remark \ref{remark:ODE1}): 
 \begin{equation}
X(r, \theta ) = \left(\displaystyle\int^r_{r_0} \cfrac{as^2}{\Delta (s)}\,ds+r\cos\theta ,r\sin\theta , \displaystyle\int ^r_{r_0}\cfrac{1}{\Delta (s)}\,ds \right),\quad r \in I ,\; \theta \in \mathbb{R}/(2\pi \mathbb{Z}), \nonumber
\end{equation}
where $\Delta(s):=\sqrt{a^2s^4+bs^2+1}$, $a>0$, $b \in \mathbb{R}$, and $I$ is a certain interval in $\mathbb{R}$. They studied causal characters of these surfaces. In this paper, we call surfaces as above {\it general type surfaces}. There may be several maximal surfaces,  timelike minimal surfaces, and surfaces which have more than one causal character in the above class. In this paper, we classify all ZMC surfaces of Riemann type according to their causal characters. More precisely, we prove the following result (Theorem \ref{thm:causalofspacelike} (1)):\begin{theor}
Causal characters of the above surface $X$ and the maximal interval of $r$ are determined as follows:
\begin{enumerate}
\renewcommand{\labelenumi}{(\roman{enumi})}
\item $X$ is spacelike if and only if $2a<b$.
\item $X$ is timelike if and only if $b< -2a$ and $0< r^2<((-b-\sqrt{b^2-4a^2})/2a^2)$.
\item $X$ has both spacelike part and lightlike part and it does not have timelike part if and only if $b=2a$. Moreover, in this case the lightlike part of the surface is a part of a straight line.
\item $X$ has both timelike part and lightlike part and it does not have spacelike part if and only if $b=-2a$ and $0<r^2<1/a$. Moreover, in this case the lightlike part of the surface is a part of a straight line.
\item The surfaces except those stated above are surfaces which have all causal characters. Moreover, in the case that $b=-2a$ and $r^2>1/a$, the lightlike part of the surface consists of a part of a straight line and two non-degenerate null curves.
\end{enumerate}
\end{theor} 
In addition to the above surfaces, we have found the following new ZMC surfaces of Riemann type which appear as solutions with constant radii of a system of ODEs, which we call {\it singular type surfaces} (see Remark \ref{remark:ODE1}):
\begin{equation*}
X(u, \theta) =\left(u+\frac{1}{\sqrt{a}}\cos \theta , \frac{1}{\sqrt{a}}\sin \theta , u\right),\quad\left(u, \theta \right) \in \mathbb{R}\times \mathbb{R}/(2\pi \mathbb{Z}),\;a>0,
\end{equation*}
and prove that these surfaces have timelike part and two lightlike lines (Theorem \ref{thm:causalofspacelike}.~(2)). In \cite{Lopez1}, L\'opez also constructed ZMC surfaces foliated by hyperbolas and parabolas which appear as solutions with non-constant radii of systems of ODEs, which we call {\it general type surfaces} in this paper. In addition to general type surfaces, we also have found ZMC surfaces foliated by hyperbolas and parabolas which appear as solutions with constant radii of systems of ODEs, and we call these surfaces {\it singular type surfaces}. In Theorems \ref{thm:causaloftimelike} and \ref{thm:causaloflightlike}, we determine causal characters of general and singular type surfaces. 
In particular, as a corollary of these classification theorems, we obtain the following corollary for ZMC surfaces of Riemann type which have exactly two causal characters.
\begin{corol}[Corollary \ref{thm:2-causal}]\label{CorB}
If a ZMC surface of Riemann type has exactly two causal characters, then the lightlike part is a part of a straight line.
\end{corol}
The class of ZMC surfaces containing lightlike lines has been studied in recent years. In \cite{FujimoriETAL1}, the authors defined the characteristics of ZMC surfaces along a lightlike line, and categorized these surfaces into the following six classes:
\begin{center}
$ \alpha^+$, $ \alpha^0_{\rm I}$, $ \alpha^0_{\rm II}$, $ \alpha^-_{\rm I}$, $ \alpha^-_{\rm II}$, $ \alpha^-_{\rm III}$,
\end{center}
and constructed some examples belonging to the above classes (see Table 1 and 2). However, as mentioned above, ZMC surfaces containing lightlike lines have not been sufficiently discussed. There are not so many known examples of ZMC surfaces containing lightlike lines. In Section 5, we determine the characteristics of ZMC surfaces of Riemann type containing lightlike lines which appear in Corollary \ref{CorB} above and the one exceptional case (Theorem \ref{thm:sp} (1) (iii)) and show that there are one- or two-parameter families of ZMC surfaces of Riemann type which have a lightlike line with characteristics $ \alpha^+$, $ \alpha^-_{\rm I}$, $ \alpha^-_{\rm II}$, and $ \alpha^-_{\rm III}$.

In Section 6, we consider ZMC entire graphs which have all causal characters. Calabi \cite{Calabi} showed that there are no maximal entire graphs except for linear functions.  However, Kobayashi \cite{Kobayashi} found ZMC entire graphs which are called {\it the helicoid of the second kind} and {\it Scherk's surface of the first kind} (see (\ref{eq:6.1}) and (\ref{eq:6.2}), respectively). These ZMC entire graphs have all causal characters. After that, in \cite{ST}, Sergienko and Tkachev produced several interesting ZMC entire graphs which admit some isolated singularities. In this paper, we obtain the following new ZMC entire graph of Riemann type which is foliated by parabolas and has all causal characters:
\begin{equation}\label{entiregraph_intro}
 X(u,v) =\left(v, -e^{-4u}+u+\cfrac{v^2}{2}, -e^{-4u}-u+\frac{v^2}{2}\right),\quad (u, v)\in \mathbb{R}^2.
 \end{equation}
Most recently, this surface was also constructed independently by Fujimori, Kawakami, Kokubu, Rossman, Umehara and Yamada \cite{FujimoriETAL4} from a different viewpoint.

The organization of the paper is as follows. In Section 2, we first introduce basic notations and local theory of surfaces in $\mathbb{L}^3$. The most important notion is the causal characters of a surface (Definition \ref{def:causal}). In Section 3, we review the parametrizations of ZMC surfaces of Riemann type foliated by Euclidean circles, hyperbolas, and parabolas, which were constructed mainly by L\'opez in \cite{Lopez1} and give new ZMC surfaces of Riemann type which we call singular type surfaces (Theorems \ref{Theorem:Parameterizations of ZMC surfaces Foliated by circles}, \ref{Theorem:Parameterizations of ZMC surfaces Foliated by hyperbolas}, and \ref{thm:lightlike}). In Section 4, we determine the causal characters of all ZMC surfaces of Riemann type (Theorems \ref{thm:causalofspacelike}, \ref{thm:causaloftimelike}, and \ref{thm:causaloflightlike}) and prove Corollary \ref{CorB} above. Moreover, in this section, we explain the relationship between general type and singular type surfaces. In Section 5, we review the definition of the characteristic of a ZMC surface containing a lightlike line, which was introduced in \cite{FujimoriETAL1} and give a useful lemma for the calculation of characteristics (Lemma \ref{lemma:characteristic}). Then, we determine the characteristics of ZMC surfaces of Riemann type containing a lightlike line (Theorems \ref{thm:sp}, \ref{thm:extimelike}, and \ref{thm:exlightlike}). In Section 6, we prove that the surface (\ref{entiregraph_intro}) is a ZMC entire graph of Riemann type which has all causal characters (Theorem \ref{Therem:6}).
\\

\section{Notations and Preliminaries\label{preliminary}}
 We denote by $\mathbb{L}^3$ the three-dimensional Lorentz-Minkowski space, that is, the three-dimensional real vector space $\mathbb{R}^3$ with the metric 
 \begin{center}
$\langle \ ,\  \rangle =dx^2+dy^2-dt^2$,
\end{center}
where $(x, y, t)$ are the canonical coordinates in $\mathbb{R}^3$. In $\mathbb{L}^3$, a vector $v$ has one of the three {\it causal characters}: it is {\it spacelike} if $\langle v, v \rangle > 0$ or $v=0$, {\it timelike} if  $\langle v, v \rangle < 0$, and {\it lightlike} if $\langle v, v \rangle = 0$ and $v\neq 0$. Let $\Sigma :=\Sigma ^2$ be a two-dimensional connected smooth manifold and $X:\Sigma \longrightarrow$ ($\mathbb{L}^3$, $\langle \ ,\  \rangle$) be a smooth map which is an immersion on an open dense subset $W\subset\Sigma$. Causal characters of a surface are defined as follows.
\begin{definition} \label{def:causal}
A surface $X$ is said to be {\it spacelike}, {\it timelike}, or {\it lightlike} on a subset $U$ in $W$ if the induced metric on $U$ is positive definite, non-degenerate with index 1, or degenerate on each tangent plane, respectively. This property of the subset $U$ is called the {\it causal character} of $X$ on $U$.
 \end{definition}
The causal character of a surface $X$ at each point is determined by the causal character of its normal direction, that is, $X$ is spacelike, timelike, or lightlike at a point $p$ if and only if the normal direction is timelike, spacelike, or lightlike, respectively. 

We denote by $\nabla$ and $\overline{\nabla}$ the Levi-Civita connections on $\Sigma$ and $\mathbb{L}^3$, respectively. On the spacelike or timelike part of $X$, the second fundamental form ${\rm II}$ and the mean curvature vector $\overrightarrow{H}$ are defined as 
\begin{align}
{\rm II}(Y,Z)&=\overline{\nabla}_{Y} Z-\nabla_{Y} Z, \nonumber \\
\overrightarrow{H}&=(1/2)\tr_{ds^2}{{\rm II}}, \nonumber
\end{align}
where $Y$ and $Z$ are smooth vector fields on $\Sigma $. Moreover, for the unit normal vector field $\nu$ along $X$, the mean curvature $H$ of $X$ is defined so that $\overrightarrow{H}=H\nu$ holds. If we take a local coordinate system $(U;u,v)$ on $\Sigma $\, then the induced metric can be written as $ds^2=Edu^2+2Fdudv+Gdv^2$ on $U$, where $E:=\langle X_u$, $X_u \rangle $, $F:=\langle X_u, X_v \rangle $, $G:=\langle X_v, X_v \rangle$, $X_u=\partial X/\partial u$ and $X_v=\partial X/\partial v$. The unit normal vector field $\nu $ on $U$ is given by
\begin{equation}\label{eq:normal direction}
\nu =\cfrac{ X_u \times X_v }{\sqrt{\epsilon (EG-F^2)}}, 
\end{equation}
where 
\begin{align}
\epsilon =\left\{ \begin{array}{ll}
1, & \text{if the surface is spacelike}, \nonumber \\
-1, & \text{if the surface is timelike}, \\
\end{array} \right.
\end{align}
and $\times$ stands for the Lorentzian vector product in $\mathbb{L}^3$. Moreover, by equation (\ref{eq:normal direction}), the normal vector $X_u \times X_v$ satisfies
\begin{equation*}\label{EG-F^2}
\langle X_u \times X_v, X_u \times X_v  \rangle =-EG+F^2.
\end{equation*}
Hence, the causal character of the surface $X$ is also determined by the determinant of the induced metric $ds^2=X^*\langle \ ,\  \rangle$ on each local coordinate system. Set $L:=\langle X_{uu}, \nu \rangle$, $M:=\langle X_{uv}, \nu \rangle$, and $N:=\langle X_{vv}, \nu \rangle$. Then the mean curvature $H$ with respect to $\nu $ is expressed on $U$ by 
 \begin{center}
 $H=-\cfrac{\epsilon}{2} \cfrac{EN-2FM+GL}{EG-F^2}.$
 \end{center}

A surface is called a {\it ZMC surface} if its mean curvature vanishes identically on the spacelike part and the timelike part. By the above equation, the equation $H=0$ is independent of the causal characters of the surface. Moreover, if the surface is described locally as the graph of a function $f$ on the $xy$-plane, the equation $H=0$ is equivalent to \label{meancurvature}
\begin{equation}
(1-f_y^2)f_{xx}+2f_xf_yf_{xy}+(1-f_x^2)f_{yy}=0, \label{eq:ZMC}
\end{equation}
which is called the {\it zero mean curvature equation}. Since the normal direction of the graph $t=f(x,y)$ is given by the vector $(f_x,f_y,1)$, then the immersion is spacelike, timelike, or lightlike at a point $p$ if and only if $1-f_x^2-f_y^2$ is positive, negative, or zero at $p$, respectively. These conditions correspond to the conditions that the partial differential equation  (\ref{eq:ZMC}) is elliptic, hyperbolic, or parabolic, respectively.\\

 \section{ZMC surfaces of Riemann type\label{Riemtype}}
   In \cite{Riemann}, Riemann constructed non-rotational minimal surfaces in the three-dimensional Euclidean space $\mathbb{R}^3$ which are foliated by circles and straight lines in parallel planes. They are called Riemann's minimal surfaces. In this section, we define ZMC surfaces of Riemann type in $\mathbb{L}^3$ and review the parametrizations of them. The following construction is mainly due to L\'opez \cite{Lopez1} and Honda-Koiso-Tanaka \cite{HKT}. Especially, the representation formula (\ref{eq:splike}) in Theorem \ref{Theorem:Parameterizations of ZMC surfaces Foliated by circles} is a special case of \cite[Proposition 3]{HKT}, and the representation formulas (\ref{eq:type1}), (\ref{eq:type2}) in Theorem \ref{Theorem:Parameterizations of ZMC surfaces Foliated by hyperbolas} and the representation formula (1) in Theorem \ref{thm:lightlike} were constructed in \cite[Theorem 1]{Lopez1}.
\subsection{Circles in $\mathbb{L}^3$}
We will give definitions of circles in $\mathbb{L}^3$ and rotational (resp. non-rotational) surfaces in $\mathbb{L}^3$. Then, by using them, we will define ZMC surfaces of Riemann type. 
\begin{definition}[{\cite[Definition 2]{LLR}}]
Let $G=G(l)$ be the identity component of the group of isometries in $\mathbb{L}^3$ that fixes pointwise a straight line $l$. The nonlinear $G(l)$-orbit of each $p_0 \in \mathbb{L}^3 \setminus l $ is called a {\it circle} in $\mathbb{L}^3$.
\end{definition}
\begin{definition}[{\cite[Section 3]{Kobayashi} or \cite[Section 2]{Lopez1}}]
A surface is said to be {\it rotational} (resp. {\it non-rotational}) with axis $l$ if the surface is (resp. is not) invariant under the group $G$ that fixes pointwise the straight line $l$ in the above definition.
\end{definition}
\begin{definition} 
A {\it ZMC surface of Riemann type} is a non-rotational ZMC surface which is foliated by circles in $\mathbb{L}^3$ and at most countably many straight lines in parallel planes.
\end{definition}
\begin{remark}\label{remark:parallel planes}
In the above definition, we can remove the assumption that the circles are in parallel planes. In fact, for any ZMC surface foliated by circles in $\mathbb{L}^3$, these circles automatically lie in parallel planes (see \cite[Theorem 4]{LLR} and \cite[Proposition 2]{Lopez1}).
\end{remark}
Let $\{e_1, e_2, e_3\} $ be the canonical base in $\mathbb{R}^3$, that is, $e_1:=(1,0,0), e_2:=(0,1,0), e_3:=(0,0,1)$. Circles in $\mathbb{L}^3$ are classified into the following three types up to isometries in $\mathbb{L}^3$ according to causal characters of the fixed straight line. 
\begin{enumerate}
  \item {\it The case that $l$ is timelike}. We may assume that $l:=\text{span}\{e_3\}$, where `\text{span}' denotes the set of linear combinations. Then every element of the connected component containing the identity of the isometry group which fixes the line $l$ is written as 
\begin{center}
$  A= \left(
    \begin{array}{ccc}
      \cos\theta & -\sin\theta  & 0 \\
      \sin\theta  & \cos\theta  & 0 \\
      0 & 0 & 1 
 \end{array}
  \right) $.
  \end{center}
  In this case, the circles are Euclidean circles written as 
  \begin{equation*}
  \alpha(\theta)=c+r\cos\theta e_1+r\sin\theta e_2,\quad c \in l,\;r\neq0. 
  \end{equation*}
  Each circle lies in a timelike plane which is parallel to the $xy$-plane.
  \item {\it The case that $l$ is spacelike}. We may assume that $l:=\text{span}\{e_1\}$. Then every element of the connected component containing the identity of the isometry group which fixes the line $l$ is written as 
\begin{center}
$  A= \left(
    \begin{array}{ccc}
      1 & 0 & 0 \\
      0  & \cosh\theta  & \sinh\theta \\
      0 & \sinh\theta & \cosh\theta
 \end{array}
  \right) $.
  \end{center}
 In this case, the circles are divided into two types of hyperbola: 
\begin{align*}
  {\rm Type\ I}\hspace{0.5cm}  \alpha (\theta)=c+r\cosh\theta e_2+r\sinh\theta e_3,   \\
 {\rm Type\ II}\hspace{0.5cm}  \alpha (\theta)=c+r\sinh\theta e_2+r\cosh\theta e_3,
\end{align*}
where $ c \in l, r\neq0 $. Each circle lies in a timelike plane which is parallel to the $yt$-plane.
  \item {\it The case that $l$ is lightlike}: We may assume that $l:=\text{span}\{e_2+e_3\}$. Then every element of the connected component containing the identity of the isometry group which fixes the line $l$ is written as 
\begin{center}
$  A= \left(
    \begin{array}{ccc}
      1 & \theta & -\theta \\
      -\theta  & 1-\frac{\theta^2}{2}  & \frac{\theta^2}{2} \\
      -\theta & -\frac{\theta^2}{2} & 1+\frac{\theta^2}{2}
 \end{array}
  \right) $.
  \end{center}
In this case, the circles are parabolas written as 
  \begin{equation*}
  \alpha (s)=c+se_1+\frac{rs^2}{2}(e_2+e_3),\quad r\neq0.
    \end{equation*}
Each circle lies in a lightlike plane which is parallel to the plane $\text{span}\{e_1, e_2+e_3\}$.
\end{enumerate}

\subsection{ZMC surfaces of Riemann type which are foliated by Euclidean circles}The surfaces foliated by circles in spacelike parallel planes are written as 
 \begin{equation}\label{eq:1}
X(u,\theta) =(f(u)+r(u)\cos \theta, g(u)+r(u)\sin \theta, u),\quad u \in J,\; \theta \in \mathbb{R}/(2\pi \mathbb{Z}), 
\end{equation}
where $f$, $g$, $r$ are smooth functions in $u$ defined on some interval $J$ satisfying $r>0$.
\begin{theorem}\label{Theorem:Parameterizations of ZMC surfaces Foliated by circles}
  The surface (\ref{eq:1}) is a ZMC surface of Riemann type if and only if the surface is congruent to the following
parametrized surface.
\begin{enumerate}
 \item General type:
  \begin{equation}
X(r, \theta ) = \left(\displaystyle\int^r_{r_0} \cfrac{as^2}{\Delta (s)}\,ds+r\cos\theta ,r\sin\theta , \displaystyle\int ^r_{r_0}\cfrac{1}{\Delta (s)}\,ds \right),\quad r, r_0\in I ,\; \theta \in \mathbb{R}/(2\pi \mathbb{Z}),\label{eq:splike}
\end{equation}
where $\Delta(s):=\sqrt{a^2s^4+bs^2+1}, \  a >0$, $b \in \mathbb{R}$, and $I$ is a certain interval in $\mathbb{R}$. \\
The condition $a \neq 0$ is equivalent to the condition that the surface is non-rotational. The interval $I$ is determined by the condition that $\Delta$ is positive.
\item Singular type:
\begin{equation}
X(u, \theta) =\left(u+\frac{1}{\sqrt{a}}\cos \theta , \frac{1}{\sqrt{a}}\sin \theta , u\right),\quad\left(u, \theta \right) \in \mathbb{R}\times \mathbb{R}/(2\pi \mathbb{Z}), \label{eq:Singular_splike}
\end{equation}
where $ a >0$.
\end{enumerate}
\end{theorem}
 \begin{remark} \label{remark:ODE1}
The condition that the surface written as (\ref{eq:1}) is a ZMC surface if and only if, after a suitable translation and rotation around the $t$-axis, the functions $r$, $f$, and $g$ satisfy the following system of ODEs (see the equations (59) and (60) in the proof of \cite[Proposition 3]{HKT}):
\begin{align}
f'=ar^2, \nonumber\\
g=0,\nonumber\\
a^2r^4-1-rr''+r'^2=0.\nonumber
\end{align}
In Theorem \ref{Theorem:Parameterizations of ZMC surfaces Foliated by circles}, general type and singular type surfaces correspond to solutions of the system of ODEs above with non-constant $r$ and constant $r$, respectively. In (\ref{eq:1}), we obtained singular type surfaces by taking a parameter $u$ of the function $r$ instead of taking $r$ as an independent variable.
 \end{remark}

\subsection{ZMC surfaces of Riemann type foliated by circles in timelike parallel planes}
The surfaces foliated by circles in timelike parallel planes are written as
  \begin{align}
  {\rm Type\ I}\hspace{0.5cm}  X(u,\theta) =(u, f(u)+r(u)\cosh \theta, g(u)+r(u)\sinh \theta), \label{eq:2} \\
 {\rm Type\ II}\hspace{0.5cm} X(u,\theta) =(u, f(u)+r(u)\sinh \theta, g(u)+r(u)\cosh \theta), \label{eq:3}
\end{align}
where $u \in J$, $v \in \mathbb{R}$, and $J$ is an open interval on $\mathbb{R}$. $r$, $f$, and $g$ are functions on $J$. 
\begin{theorem}\label{Theorem:Parameterizations of ZMC surfaces Foliated by hyperbolas}
 The surfaces (\ref{eq:2}) and (\ref{eq:3}) are ZMC surfaces of Riemann type if and only if the surfaces are congruent to the following
parametrized surfaces, respectively:
  \begin{enumerate}
  \item General type:
    \begin{align}
{\rm Type\ I}\hspace{0.5cm}X(r,\theta)=\left(
\displaystyle\int^r_{r_0} \cfrac{1}{\Delta (s)}\,ds, \int^r_{r_0} \cfrac{as^2}{\Delta (s)}\,ds, \int^r_{r_0} \cfrac{bs^2}{\Delta (s)}\,ds \right) +\left( 0, r\cosh\theta, r\sinh\theta
 \right),  \label{eq:type1}  
 \end{align}
where $r\in I$, $\theta \in \mathbb{R}$, and $\delta$, $a$, $b\in \mathbb{R}$, $(a,b) \neq (0,0)$, $\Delta(s):=\sqrt{(a^2-b^2)s^4+2\delta s^2-1}$. 

 \begin{align}
{\rm Type\ II}\hspace{0.5cm}X(r,\theta)=\left(
\displaystyle\int^r_{r_0} \cfrac{1}{\Delta (s)}\,ds, \int^r_{r_0} \cfrac{as^2}{\Delta (s)}\,ds, \int^r_{r_0} \cfrac{bs^2}{\Delta (s)}\,ds \right) +\left( 0, r\sinh\theta, r\cosh\theta
 \right), \label{eq:type2}
 \end{align}  
where $r\in I$, $\theta \in \mathbb{R}$, and $\delta$, $a$, $b\in \mathbb{R}$, $(a,b) \neq (0,0)$, $\Delta(s):=\sqrt{(-a^2+b^2)s^4-2\delta s^2+1}$. \\
 The condition $(a,b) \neq (0,0)$ is equivalent to the condition that the surface is non-rotational. The interval $I$ is determined by the condition that $\Delta$ is positive.
\item Singular type:
\begin{align}
{\rm Type\ I}\hspace{0.5cm}X(u, \theta) =\frac{u}{c} \left(c, a, b\right) +\frac{1}{\sqrt{c}}\left( 0, \cosh\theta, \sinh\theta \right), \label{eq:Singular_hyperbpla1} \\
{\rm Type\ II}\hspace{0.5cm}X(u, \theta) =\frac{u}{c} \left(c, a, b\right) +\frac{1}{\sqrt{c}}\left( 0, \sinh\theta, \cosh\theta \right), \label{eq:Singular_hyperbpla2}
\end{align}
where $-a^2+b^2>0$ and $c:=\sqrt{-a^2+b^2}$.
\end{enumerate}
\end{theorem}

 \begin{remark} \label{remark:ODE2}
The surface (\ref{eq:2}) and the surface (\ref{eq:3}) are ZMC surfaces if and only if the functions $r$, $f$, and $g$ satisfy the following systems of ODEs, respectively (see the equations (24)--(26) and (28)--(30) in the proof of \cite[Theorem 1]{Lopez1}):
\begin{align}
f'=ar^2, \nonumber\\
g'=br^2,\nonumber\\
1+(a^2-b^2)r^4+r'^2-rr''=0,\nonumber
\end{align}
and
\begin{align}
f'=ar^2, \nonumber\\
g'=br^2,\nonumber\\
1+(a^2-b^2)r^4-r'^2+rr''=0.\nonumber
\end{align}
In Theorem \ref{Theorem:Parameterizations of ZMC surfaces Foliated by hyperbolas}, general type and singular type surfaces correspond to solutions of the systems of ODEs above with non-constant $r$ and constant $r$, respectively. In (\ref{eq:2}) and (\ref{eq:3}), we have obtained singular type surfaces by taking a parameter $u$ of the function $r$ instead of taking $r$ as an independent variable.
\end{remark}
 
\begin{remark}
As well as the parametrizations (\ref{eq:splike}) and (\ref{eq:Singular_splike}), we can normalize the constants $a$ and $b$ in (\ref{eq:type1})--(\ref{eq:Singular_hyperbpla2}) as follows. When $|a|\neq |b|$, after an isometry in $\mathbb{L}^3$ which fixes the axis $l=\text{span}\{e_1\}$ and a reparametrization of the surface $X$, we can take $a=0$ or $b=0$ depending on the relation $|a|<|b|$ or $|a|>|b|$. For the case $|a|=|b|$, we can also normalize $(a,b)=(1,1)$. 
\end{remark}

 \subsection{ZMC surfaces of Riemann type foliated by circles in lightlike parallel planes}
  The surfaces foliated by circles in lightlike parallel planes are written as 
  \begin{equation}X(u,v) =\left(f(u)+v, g(u)+u+\cfrac{r(u)v^2}{2}, g(u)-u+\cfrac{r(u)v^2}{2}\right), \label{eq:4}
\end{equation}
where $u \in J$, $v \in \mathbb{R}$, and $f$, and $g$ are functions on $J$, and $r$ is a function which has no zeros. $J$ is the set of open intervals in $\mathbb{R}$ where all of the functions $r$, $f$, and $g$ in Theorem \ref{thm:lightlike} are well-defined.
\begin{theorem}\label{thm:lightlike}
  The surface (\ref{eq:4}) is a ZMC surface of Riemann type if and only if the triple \{$r$, $f$, $g$\} satisfies one of the following four conditions.
    \begin{itemize}
\item[(1)] General type:
  \begin{itemize}
 \item[(i)] $r(u)=\cfrac{1}{-2u+c}$, \\
$f(u)=\cfrac{4b}{3}u^3-2bcu^2+bc^2u $, \\
$g(u)=\displaystyle \int^u_0 (b^2(s^2-cs)+p)(c-2s)^2ds$,\\
$b\in \mathbb{R}\setminus \{0\}, c, p \in \mathbb{R}$.\\
The condition $b\neq 0$ is equivalent to the condition that the surface is non-rotational (see Proposition \ref{A_Prop} in Appendix \ref{A}).

\item[(ii)] $r(u)=\sqrt{\cfrac{a}{2}}\tan\{ \sqrt{2a} (u+c) \}$, \\
$f(u)=-\cfrac{\sqrt{2}b}{a^{\frac{3}{2}}}\cot\{ \sqrt{2a} (u+c) \}-\cfrac{2b}{a}(u+c)$, \\
$g(u)=-\cfrac{b^2}{\sqrt{2}a^{\frac{5}{2}}}\cot\{ \sqrt{2a} (u+c) \}+\cfrac{p}{4\sqrt{2a}}\sin\{ 2\sqrt{2a}(u+c) \} +(\cfrac{p}{2}-\cfrac{b^2}{a^2})(u+c)$, \vspace{0.1cm} \\
$a>0, b, c, p \in \mathbb{R}$.  \\

\item[(iii)] $r(u)=\sqrt{\cfrac{-a}{2}}\tanh\{ \sqrt{-2a} (-u+c) \}$, \\
$f(u)=-\cfrac{\sqrt{2}b}{a\sqrt{-a}}\coth\{ \sqrt{-2a} (-u+c) \}-\cfrac{2b}{a}(u+c)$, \\
$g(u)=-\cfrac{b^2}{a^2 \sqrt{-2a}}\coth\{ \sqrt{-2a} (-u+c) \} -\cfrac{p}{4\sqrt{-2a}}\sinh\{ 2\sqrt{-2a}(-u+c) \} +(\cfrac{b^2}{a^2}-\cfrac{p}{2})(-u+c)$, \vspace{0.1cm} \\
$a<0, b, c, p \in \mathbb{R}$.  \\
\end{itemize}

\item[(2)] Singular type: \\
$r(u)=\sqrt{\cfrac{-a}{2}}$, \\
$f(u)=-\cfrac{2b}{a}u $, \\
$g(u)=pe ^{-2\sqrt{-2a}u}-\cfrac{b^2}{a^2}u$,  \vspace{0.1cm} \\
$a<0, b, p \in \mathbb{R}$.\\
\end{itemize}
\end{theorem}
\begin{remark} \label{remark:ODE3}
The functions $r$, $f$, and $g$ are the solutions of following ODEs (see the equations (35)--(38) in the proof of \cite[Theorem 1]{Lopez1}):
\begin{align}
r'=2r^2+a \label{eq:r}, \\
r^2f'=b \label{eq:f}, \\
\cfrac{b^2}{r^3}+4rg'+g''=0. \label{eq:g}
\end{align}
(\ref{eq:r}), (\ref{eq:f}), and (\ref{eq:g}) are equivalent to the condition that the surface $X$ is a ZMC surface, and the cases (i), (ii), and (iii) are corresponding to the cases $a=0$, $a>0$, and $a<0$ in the equation ({\ref{eq:r}}). In Theorem \ref{thm:lightlike}, general type and singular type surfaces correspond to solutions of the system of ODEs above with non-constant $r$ and constant $r$, respectively. In (\ref{eq:4}), we have obtained singular type surfaces by taking a parameter $u$ of the function $r$ instead of taking $r$ as an independent variable.
 \end{remark} 
\begin{remark} 
In Theorem \ref{thm:lightlike}, we used different integral constants from constants in Theorem 1 in \cite{Lopez1} so that our expression is more convenient to check the causal characters of surfaces. We also corrected some typographical errors in \cite{Lopez1}. 
 \end{remark} 
 \section{Causal characters of ZMC surfaces of Riemann type\label{Maintheorem}}
In this section we classify ZMC surfaces of Riemann type constructed by the previous section by their causal characters.
 
 \subsection{Causal characters of ZMC surfaces of Riemann type foliated by Euclidean circles\label{circle}}
\begin{theorem}\label{thm:causalofspacelike}
Causal characters of ZMC surfaces given by (\ref{eq:splike}) or (\ref{eq:Singular_splike}) and the maximal interval of $r$ are determined as follows.
\begin{enumerate}
\item General type:
\begin{itemize}
\item[(i)] The surface $X$ is spacelike if and only if $2a<b$.
\item[(ii)] The surface $X$ is timelike if and only if 
\begin{center}
$b< -2a$ and $0<r^2<((-b-\sqrt{b^2-4a^2})/2a^2)$.\end{center}
\item[(iii)] The surface $X$ has both a spacelike part and a lightlike part and it does not have a timelike part if and only if $b=2a$. Moreover, in this case the lightlike part of the surface is a part of a straight line.
\item[(iv)] The surface $X$ has both a timelike part and a lightlike part and it does not have a spacelike part if and only if $b=-2a$ and $0<r^2<1/a$. Moreover, in this case the lightlike part of the surface is a part of a straight line.
\item[(v)] The surfaces except those stated above are surfaces which have all causal characters. Moreover, in the case that $b=-2a$ and $r^2>1/a$, the lightlike part of the surface consists of a part of a straight line and two non-degenerate null curves.
\end{itemize}
\item Singular type: \\
Every singular type surface has timelike part and two lightlike lines and it does not have spacelike part.
\end{enumerate}
\end{theorem} 
\begin{proof} 
 First we prove (1). (i) and (ii) are special cases of (II) and (III) of \cite[Proposition 3]{HKT}, respectively. Now, we consider the other cases. Since the normal vector of the surface (\ref{eq:splike}) is 
 \[
 X_r \times X_\theta =\left(-\frac{r\cos \theta}{\Delta (r)}, -\frac{r\sin \theta}{\Delta (r)}, -r-\frac{ar^3\cos \theta}{\Delta (r)}\right)
 \], the determinant of the metric with respect to the local coordinate system $(r,\theta)$ is
\begin{equation}
EG-F^2=-\cfrac{r^2}{\Delta (r)^2} \left(1+ar^2\cos \theta +\Delta (r) \right) \left(1-ar^2\cos \theta -\Delta (r) \right). \label{eq:det}
\end{equation}
Set $\xi(r, \theta):=1+ar^2\cos \theta +\Delta (r)$ and $\eta(r, \theta):=1-ar^2\cos \theta -\Delta (r)$.

First, we determine the condition that a surface $X$ has both a spacelike part and a lightlike part. Since $\xi$ and $\eta$ are continuous functions on the connected domain $I \times \mathbb{R}$ and $\xi +\eta =2$ holds, if $\xi$ or $\eta$ changes its sign, $EG-F^2$ also changes its sign, and  $\xi >0$ for some point of the domain $I \times \mathbb{R}/(2\pi \mathbb{Z})$. Therefore, the condition that $EG-F^2\geq 0$ on $I \times \mathbb{R}$ is equivalent to the following:
\begin{equation} 
\xi \geq0\text{ and }\eta \leq0\text{ on }I \times \mathbb{R}. \label{eq:xi and eta}
\end{equation}
The first condition is equivalent to $\Delta (r) >ar^2-1$ on $I$, and the second one is equivalent to $\Delta (r) >ar^2+1$ on $I$, that is, (\ref{eq:xi and eta}) is equivalent to $b\geq2a$. In this case, $\Delta$ is positive for every $r\in \mathbb{R}$, and therefore we obtain $I=\mathbb{R}$. Since the condition $b>2a$ is the spacelike condition (i), the surface has both a spacelike part and a lightlike part if and only if $b=2a$. Next, we determine the lightlike part of the surface. When $b=2a$, $\Delta=ar^2+1$ and so $I=\mathbb{R}$. In this case we can take $r_0=0$ in the equation (\ref{eq:splike}). By (\ref{eq:det}), 
\begin{equation}
EG-F^2=\cfrac{r^2}{\Delta (r)^2} \left(2+ar^2\cos \theta +ar^2\right) \left(ar^2\cos \theta +ar^2\right).  \nonumber
\end{equation}
Since $ar^2\cos \theta +ar^2\geq 0$, the surface has the following lightlike part:
\begin{equation*}
c(r)=(x(r), y(r), z(r))=X(r, \pi) = \left(\displaystyle\int^r_0 \cfrac{as^2}{as^2+1}\,ds-r, 0, \displaystyle\int ^r_0\cfrac{1}{as^2+1}\,ds \right),
\end{equation*}
Since $x(r)+z(r)=0$, the lightlike part is a part of a straight line. The proof of (iii) has been completed.

 Second, we determine the condition that a surface $X$ has both timelike part and lightlike part. The condition that $EG-F^2\geq 0$ on $I \times \mathbb{R}$ is equivalent to the following:
\begin{equation} 
\xi \geq0\text{ and }\eta \geq0\text{ on }I \times \mathbb{R}. \label{eq:xi and eta2}
\end{equation}
The first condition is equivalent to $\Delta (r) \geq ar^2-1$ on $I$, and the second one is equivalent to $1-ar^2\geq \Delta (r)$ on $I$, that is, (\ref{eq:xi and eta2}) is equivalent to $b\leq -2a$ and $r^2<\frac{1}{a}$. By considering the condition $\Delta >0$, this is equivalent to $b\leq -2a$ and $r^2<((-b-\sqrt{b^2-4a^2}/2a^2)$. Since the condition $b<-2a$ is the timelike condition (ii), the surface has both a timelike part and a lightlike part if and only if  $b=-2a$ and $r^2<1/a$. In this case we can take $r_0=0$ in the equation (\ref{eq:splike}). Next, we determine the lightlike part of the surface. In this case, since $\Delta=1-ar^2$, (\ref{eq:det}) becomes
\begin{equation}
EG-F^2=-\cfrac{r^2}{\Delta (r)^2} \left(2+ar^2\cos \theta -ar^2\right) \left(-ar^2\cos \theta +ar^2\right) . \nonumber
\end{equation}
Since $2+ar^2\cos \theta -ar^2\geq 2(1-ar^2)>0$ and $-ar^2\cos \theta +ar^2\geq 0$, $EG-F^2\geq 0$ holds, that is, the surface has timelike part and the following lightlike part:
\begin{equation*}
c(r)=(x(r), y(r), z(r))=X(r,2\pi) = \left(\displaystyle\int^r_0 \cfrac{as^2}{1-as^2}\,ds+r, 0, \displaystyle\int ^r_0\cfrac{1}{1-as^2}\,ds \right),
\end{equation*}
Since $x(r)-z(r)=0$, the lightlike part is a part of a straight line. The proof of (iv) has been completed.

Third, we determine the condition that a surface $X$ has all causal characters. Since $a\neq 0$, $EG-F^2$ is not identically zero. Therefore, there is no lightlike surface. This means that surfaces except those stated above are surfaces which have all causal characters. Now, assume that the conditions $b=-2a$ and $r^2>1/a$ hold.
Since $\Delta=ar^2-1$, (\ref{eq:det}) becomes
\begin{equation}
EG-F^2=-\cfrac{r^2}{\Delta (r)^2} \left(ar^2\cos \theta +ar^2\right) \left(2-ar^2\cos \theta -ar^2\right).\nonumber
\end{equation}
Note that $\xi (r,\theta)=ar^2\cos \theta +ar^2\geq 0$ and the continuous function $\eta(r, \theta)=2-ar^2\cos \theta -ar^2$ satisfies
\begin{equation}
2(1-ar^2)\leq \eta(r, \theta)\leq 2. \nonumber
\end{equation}
Since $1-ar^2< 0$, $EG-F^2$ changes its sign. Therefore, the surface changes its causal characters.
Moreover, the surface is lightlike on the domain which satisfies either $\xi (r,\theta)= 0$, that is,
\begin{equation*}
\theta=(2k+1)\pi ,\ k \in \mathbb{Z},
\end{equation*}
or $\eta(r, \theta)=0$. When the case that $\xi (r,\theta)= 0$, the lightlike part is a part of a straight line as in the case (iv), and then the surface is timelike along the straight line. Next we consider the case that $\eta(r, \theta)=0$. Since $r\cos \theta =(2-ar^2)/ar$, the lightlike part is written as the following two curves:
\begin{equation}
c(r)=X(r, \theta (r))=\left( \displaystyle\int^r_{r_0} \cfrac{as^2}{as^2-1}\,ds+\cfrac{2-ar^2}{ar} ,\pm r\sqrt{1-\left(\cfrac{2-ar^2}{ar^2} \right) ^2} , \displaystyle\int ^r_{r_0}\cfrac{1}{as^2-1}\,ds\right) , \nonumber
\end{equation}
and these curves are not straight lines. The proof of (v) and so the proof of (1) have been completed.
 
 Now we prove (2). By a straightforward computation, the determinant of the induced metric of the surface (\ref{eq:Singular_splike}) with respect to the local coordinate system $(u, \theta)$ is $-(1/a)\sin ^2 \theta$. Therefore, the surface has a timelike part and the following two lightlike lines:
\begin{equation}
 c_+(u)=X(u, 0)=u \left( 1, 0, 1 \right )+\cfrac{1}{\sqrt{a}}\left( 1, 0, 0 \right),\quad c_-(u)=X(u, \pi)=u \left(1, 0, 1 \right )-\cfrac{1}{\sqrt{a}}\left( 1, 0, 0 \right) .\nonumber
\end{equation}
\end{proof}
At the end of this subsection, it should be noted that there is a relationship between singular type surfaces and the asymptotic behaviors of some general type surfaces. If we consider a ZMC surface of Riemann type which belongs to the class (iv) in Theorem \ref{thm:causalofspacelike}, this surface has the following implicit form:
\begin{equation*}
\left(x-t+r(t)\right)^2+y^2=r(t)^2, r(t)=\cfrac{1}{\sqrt{a}}\tanh{(\sqrt{a}t)}.
\end{equation*}
By taking limits $r\to \pm (1/\sqrt{a})$ on the surface (which correspond to $t\to \pm \infty$), we obtain 
\begin{equation*}
\left(x-t\pm \frac{1}{\sqrt{a}}\right)^2+y^2=\cfrac{1}{a}.
\end{equation*}
These surfaces are nothing but the singular type surfaces (\ref{eq:Singular_splike}) up to a translation in $\mathbb{L}^3$ and Figures 4 and 5 in Section 5.2 are pictures of the above two surfaces.

  \subsection{Causal characters of ZMC surfaces of Riemann type foliated by hyperbolas\label{hyperbola}}
\begin{theorem}\label{thm:causaloftimelike} Causal characters of ZMC surfaces of Riemann type which are given by (\ref{eq:type1}), (\ref{eq:type2}), (\ref{eq:Singular_hyperbpla1}), or (\ref{eq:Singular_hyperbpla2}) and the maximal interval of $r$ are determined as follows.
\begin{itemize}
\item[(1)] General type:
\begin{description}
\item[{\rm Type I}] Timelike on whole domain.
\item[{\rm Type II}] 
\end{description}
\begin{itemize}
\renewcommand{\labelenumi}{(\roman{enumi})}
\item[(i)] There is no Type II spacelike surface on $I \times \mathbb{R}$.
\item[(ii)] The surface $X$ is timelike if and only if  \\
$|b|\geq|a|$, and if we define $c:=\sqrt{-a^2+b^2}$, then one of the following three conditions holds.
\begin{itemize}
\item[(ii-1)] In the case when $a<b$,\\
$|\delta |<c$, $0 \leq c \leq -\delta$, or `$0<c\leq \delta$ and $r^2>\cfrac{\delta +\sqrt{\delta ^2-c^2}}{c^2}$'.
\item[(ii-2)] In the case when $a=b$,\\
$a=b>0$ and $\delta \leq 0$. 
\item[(ii-3)]In the case when $a>b$,\\
$ 0<c < \delta$ and $r^2>\cfrac{\delta +\sqrt{\delta ^2-c^2}}{c^2}$. 
\end{itemize}
\item[(iii)] There is no Type II surface which has both a spacelike part and a lightlike part and it does not have a timelike part.
\item[(iv)] The surface $X$ has both a timelike part and a lightlike part and it does not have a spacelike part if and only if
$|b|>|a|$, and if we define $c:=\sqrt{-a^2+b^2}$, then one of the following two conditions holds.
\begin{itemize}
\item[(iv-1)] In the case when $a<b$, \\
$0<c=\delta $ and $0<r^2<\cfrac{1}{c}$.
\item[(iv-2)] In the case when $a>b$, \\
$0<c=\delta $ and $r^2>\cfrac{1}{c}$.
\end{itemize}
Moreover in this case, the lightlike part of the surface is a part of a straight line.
\item[(v)] The Type II surfaces except those stated above are surfaces which have all causal characters.
\end{itemize}
\item[(2)] Singular type:
\begin{description}
\item[{\rm Type I}] Timelike on whole domain.
\item[{\rm Type II}] The surface has timelike part and a lightlike line  and it does not have spacelike part.
\end{description}
\end{itemize}

\end{theorem}
\begin{proof} 
 First we prove (1). The Type I surface $X$ is written as (\ref{eq:type1}). By a straightforward calculation, we get
\begin{center}
$X_r \times X_\theta=\left( \cfrac{ar^3\cosh\theta }{\Delta(r)}+r-\cfrac{br^3\sinh\theta}{\Delta(r)}, -\cfrac{r\cosh\theta}{\Delta(r)}, -\cfrac{r\sinh\theta}{\Delta(r)} \right)$,
\end{center}
 and then
 \begin{center}
$\langle X_r \times X_\theta$, $X_r \times X_\theta \rangle =\left( \cfrac{ar^3\cosh\theta }{\Delta(r)}+r-\cfrac{br^3\sinh\theta}{\Delta(r)}\right)^2+\left( \cfrac{r}{\Delta(r)}\right)^2 >0$. 
\end{center}
Therefore, the normal vector of the surface is spacelike, which means that the surface is timelike on whole domain.

  The {\rm Type II} surface $X$ is written as (\ref{eq:type2}). By a straightforward calculation, we get 
  \begin{center}
$X_r \times X_\theta=\left( \cfrac{ar^3\sinh\theta }{\Delta(r)}-r-\cfrac{br^3\cosh\theta}{\Delta(r)}, -\cfrac{r\sinh\theta}{\Delta(r)}, -\cfrac{r\cosh\theta}{\Delta(r)} \right)$, 
\end{center}
and then 
\begin{align}
\langle X_r \times X_\theta, X_r \times X_\theta \rangle &=\left( \cfrac{ar^3\sinh\theta }{\Delta(r)}-r-\cfrac{br^3\cosh\theta}{\Delta(r)}\right)^2-\left( \cfrac{r}{\Delta(r)}\right)^2  \nonumber \\
&=\left( \cfrac{r}{\Delta(r)}\right)^2 \left\{\left( ar^2\sinh\theta-\Delta(r)-br^2\cosh\theta \right)^2-1 \right\} \nonumber \\
&=\left( \cfrac{r}{\Delta(r)}\right)^2 \left( ar^2\sinh\theta-br^2\cosh\theta-\Delta(r)+1 \right) \left( ar^2\sinh\theta-br^2\cosh\theta-\Delta(r)-1 \right).  \nonumber 
\end{align}
If we define a new parameter $s:=e^\theta $, then
\begin{align}&ar^2\sinh\theta-br^2\cosh\theta-\Delta(r)+1 \nonumber \\
&=\cfrac{1}{2s}\left\{ (a-b)r^2s^2+2(-\Delta(r) +1)s-(a+b)r^2 \right\}\text{, and}\nonumber \\
&ar^2\sinh\theta-br^2\cosh\theta-\Delta(r)-1 \nonumber \\
&=\cfrac{1}{2s}\left\{ (a-b)r^2s^2+2(-\Delta(r) -1)s-(a+b)r^2 \right\}. \nonumber
\end{align}
Therefore, if we define $\phi:=\phi(r, s)=(a-b)r^2s^2+2(-\Delta(r) +1)s-(a+b)r^2$ and $\psi:=\psi(r,s)=(a-b)r^2s^2+2(-\Delta(r) -1)s-(a+b)r^2$, then  
\begin{equation}\label{eq:outerproduct}
EG-F^2=-\langle X_r \times X_\theta, X_r \times X_\theta \rangle=-\left( \cfrac{r}{2s\Delta(r)}\right)^2 \phi \psi .
\end{equation}

First, we determine the condition that a Type II surface $X$ is spacelike. By (\ref{eq:outerproduct}) and $\phi=\psi +4s$, the surface $X$ is spacelike on $I \times \mathbb{R}$ if and only if 
\begin{equation}
\phi >0\text{ and }\psi <0\text{ on }I \times \mathbb{R}_{>0}. \label{eq:spacelike}
\end{equation}
Then we get $a=b$, and in this case $\phi$ and $\psi$ are as follows:
\begin{subnumcases}
{}
\phi=2(-\Delta(r) +1)s-2ar^2, \label{eq:phia=b}\\
\psi=2(-\Delta(r) -1)s-2ar^2. \label{eq:psia=b}
\end{subnumcases}
Since $\phi>0$ and $\psi<0$, we get $a<0$ and $a>0$ respectively. It is a contradiction and hence the proof of (i) has been completed.
\newline

Second, we determine the condition that a Type II surface $X$ is timelike. By (\ref{eq:outerproduct}) and $\phi=\psi +4s$, the surface $X$ is timelike on $I \times \mathbb{R}$ if and only if 
\begin{equation} 
\phi <0\text{ or }\psi >0\text{ on }I \times \mathbb{R}. \label{eq:timelike}
\end{equation}

Let us consider the condition that $\phi<0$. If we assume that $a-b>0$, then for arbitrary $r>0$, there exist a large $s>0$ such that $\phi (r, s)>0$, which is a contradiction. Therefore we obtain $a\leq b$. When $a=b(\neq 0)$, the functions $\phi$ and $\Delta$ are written as follows:
\begin{align}
\phi (r,s)=2(-\Delta(r) +1)s-2ar^2, \nonumber\\
\Delta (r)=\sqrt{-2\delta r^2+1}. \nonumber
\end{align}
Therefore the condition $\phi <0$ on $I \times \mathbb{R}_{>0}$ if and only if 
\begin{equation}
-\Delta(r) +1\leq 0\text{ on }I\text{ and }a>0, \nonumber
\end{equation}
which is equivalent to 
\begin{equation}
\delta \leq 0\text{ and }a>0. \nonumber  
\end{equation}
The maximal interval $I$ is determined by the condition $\Delta>0$. Since $\delta \leq 0$, we obtain $I=\mathbb{R}_{>0}$.

When $a<b$, $\phi <0$ on $I \times \mathbb{R}_{>0}$ if and only if for every $r \in I$, the one of the following is true:\\
\begin{subnumcases}
{}
(-\Delta(r) +1)^2+(a^2-b^2)r^4<0\text{ or}  \label{eq:case1} \\
``(-\Delta(r) +1)^2+(a^2-b^2)r^4 \geq 0\text{ and} \nonumber \\
 s_0(r):=\cfrac{(\Delta -1)-\sqrt{(-\Delta (r) +1)^2+(a^2-b^2)r^4}}{(a-b)r^2} \leq 0." \label{eq:case2}
\end{subnumcases}
If (\ref{eq:case1}) holds for a $r \in I$, then $a^2-b^2 <0$ and we can take $c:=\sqrt{-a^2+b^2}>0$. Using the constant $c$, (\ref{eq:case1}) can be written as follows:\\
\begin{equation}
(-\Delta(r) +1)^2-c^2r^4 <0, \nonumber
\end{equation}
which is equivalent to
\begin{subnumcases}{}
-\Delta(r) +1+cr^2>0\text{ and} \label{eq:case1.1} \\
-\Delta(r) +1-cr^2 <0. \label{eq:case1.2}
\end{subnumcases}
Considering the condition $\Delta>0$, (\ref{eq:case1.1}) is equivalent to 
\begin{equation}
(1+cr^2)^2>\Delta (r)^2, \nonumber 
\end{equation}
that is, we obtain $c>-\delta$.\\
(\ref{eq:case1.2}) is equivalent to
\begin{subnumcases}{ }
 1-cr^2 \leq 0\text{ or} \\
 1-cr^2>0\text{ and }1-cr^2<\Delta (r). \label{eq:case1.2.b}
  \end{subnumcases}
 In the case of (\ref{eq:case1.2.b}), taking the square of both sides of the inequality $1-cr^2<\Delta (r)$, (\ref{eq:case1.2.b}) is equivalent to $c>\delta $ and $r^2<\cfrac{1}{c}$. Therefore (\ref{eq:case1}) is equivalent to 
 \begin{subnumcases}{ }
c>-\delta \text{ and }\cfrac{1}{c}\leq r^2,\text{ or} \label{eq:case1.1.1} \\
c>-\delta , c>\delta \text{, and }r^2<\cfrac{1}{c}. \label{eq:case1.1.2} 
\end{subnumcases}
Next we determine the maximal interval $I$, which is determined by the condition that $\Delta (r)>0$. In the case of (\ref{eq:case1.1.2}), the discriminant of $\Delta (r)^2=c^2r^4-2\delta +1$ is $\delta^2-c^2<0$, then we can take $I=(-1/\sqrt {c}, 1/\sqrt {c})$. 
In the case of (\ref{eq:case1.1.1}), if $c>\delta$ we can take $I=(-\infty , -1/\sqrt {c}]$, or $[1/\sqrt {c}, \infty )$ for the same reason as above. If $c\leq \delta$, then the discriminant of $\Delta (r)^2=c^2r^4-2\delta +1$ is $\delta^2-c^2\geq 0$, and by considering the following inequality:
\begin{equation}
\cfrac{\delta -\sqrt{\delta ^2-c^2}}{c^2}\leq \cfrac{1}{c}\leq \cfrac{\delta +\sqrt{\delta ^2-c^2}}{c^2}, \nonumber
\end{equation}
(\ref{eq:case1.1.1}) is equivalent to
\[
\begin{cases}
c>|\delta |\text{ and }r^2\geq \cfrac{1}{c}, or   \\
0<c\leq \delta \text{ and }\cfrac{\delta+\sqrt{\delta ^2-c^2}}{c^2}<r^2. 
\end{cases}
\]
Combining (\ref{eq:case1.1.2}), (\ref{eq:case1}) is equivalent to
\begin{subnumcases}{}
c>|\delta |\text{ or }\label{eq:1.2.1} \\
0<c\leq \delta \text{ and }\cfrac{\delta+\sqrt{\delta ^2-c^2}}{c^2}<r^2 \label{eq:1.2.2} . 
\end{subnumcases}
\\

Next we consider the condition (\ref{eq:case2}). If we assume that $a^2-b^2>0$, by the following inequality:
\begin{equation}
\sqrt{(-\Delta (r) +1)^2+(a^2-b^2)r^4}>|1-\Delta|, \nonumber
\end{equation}
then we obtain
\begin{equation}
s_0(r)>\cfrac{(\Delta -1)-|1-\Delta|}{(a-b)r^2}\geq 0, \nonumber
\end{equation}
which is a contradiction. Therefore we get $a^2-b^2 \leq 0$, and take $c:=\sqrt{-a^2+b^2}$.
Using the constant $c$, (\ref{eq:case2}) is equivalent to the following conditions:
\begin{subnumcases}{}
(-\Delta +1) ^2-c^2r^4 \geq 0\text{ and }\label{eq:2.1} \\
(\Delta -1)-\sqrt{(-\Delta +1)^2-c^2r^4}\geq 0. \label{eq:2.2} 
\end{subnumcases}

Since $(-\Delta +1) ^2-c^2r^4=(-\Delta +1+cr^2)(-\Delta +1-cr^2) $ and $c \geq 0$, (\ref{eq:2.1}) is equivalent to 
\[
\begin{cases}{}
 1+cr^2 \leq \Delta \text{ or}  \\
1-cr^2 \geq \Delta , 
\end{cases}
\]
that is, 
\begin{subnumcases}{}
 c\leq -\delta \text{ or}\label{eq:2.1.1} \\
c \leq \delta . \label{eq:2.1.2}
\end{subnumcases}

on the other hand, clearly, (\ref{eq:2.2}) is equivalent to $\Delta (r)\geq 1$, which is 
\begin{subnumcases}{}
c=0\text{ and }\delta \leq \ 0, \text{ or }\label{eq:2.2.1} \\
c\neq 0\text{ and }\ r^2 \geq \cfrac{2\delta}{c^2}. \label{eq:2.2.2} 
\end{subnumcases}
By (\ref{eq:2.1.1}), (\ref{eq:2.1.2}), (\ref{eq:2.2.1}) and (\ref{eq:2.2.2}), (\ref{eq:case2}) is equivalent to
\begin{subnumcases}{}
0\leq c \leq -\delta \text{ or }\label{eq:2.3.1}\\
0< c \leq \delta \text{ and }r^2 \geq \cfrac{2\delta}{c^2} \label{eq:2.3.2}.
\end{subnumcases}
Therefore, by (\ref{eq:1.2.1}), (\ref{eq:1.2.2}), (\ref{eq:2.3.1}) and (\ref{eq:2.3.2}),  (\ref{eq:case1}) and (\ref{eq:case2}) are equivalent to 
\[
\begin{cases}
 c>|\delta |\text{ or}     \\
0<c\leq \delta \text{ and }\cfrac{\delta+\sqrt{\delta ^2-c^2}}{c^2}<r^2 \text{, or}   \\
0\leq c \leq -\delta \text{ or}  \\
0< c \leq \delta \text{ and}\ r^2 \geq \cfrac{2\delta}{c^2} .
\end{cases}
\]
In the case that $0< c \leq \delta$, since 
\[
\cfrac{\delta+\sqrt{\delta ^2-c^2}}{c^2}<r^2 \leq \cfrac{\delta+\sqrt{\delta ^2}}{c^2}=\cfrac{2\delta}{c^2}
\], the above condition is equivalent to \\
\begin{subnumcases}{ }
c>|\delta |\text{ or} \nonumber\\
0\leq c \leq-\delta \text{ or} \nonumber\\
0<c\leq \delta \text{ and }\cfrac{\delta+\sqrt{\delta ^2-c^2}}{c^2}<r^2.\nonumber
\end{subnumcases}
The above inequalities are the equivalent condition that $\phi <0$ and $a<b$. 

 Next, let us consider the condition that $\psi>0$. If we assume that $a-b<0$, then for arbitrary $r>0$, there exist a large $s>0$ such that $\psi (r, s)>0$, which is a contradiction. Therefore we obtain $a\geq b$. When $a=b(\neq 0)$, since $\psi=2(-\Delta(r) -1)s-2ar^2 $, the condition $\psi>0$ does not hold.

 When $a>b$, $\psi >0$ on $I \times \mathbb{R}_{>0}$ if and only if for every $r \in I$, the one of the following is true:
\[
\begin{cases}
(\Delta(r)+1)^2+(a^2-b^2)r^4<0\text{ or} \\
``(\Delta(r)+1)^2+(a^2-b^2)r^4\geq 0\text{ and}\  
\cfrac{\Delta (r)+1+\sqrt{(\Delta(r) +1)^2+(a^2-b^2)r^4}}{(a-b)r^2}\leq 0."
\end{cases}
\]

Clearly,  the latter case does not occur, and we can take $c:=\sqrt{-a^2+b^2}>0$. Since $(\Delta (r) +1)^2-c^2r^4=(\Delta (r)+1+cr^2)(\Delta(r) +1-cr^2)$, the condition that $\psi>0$ on $I\times \mathbb{R}_{>0}$ is equivalent to $-a^2+b^2>0$, $0<c<\delta $, and $r^2>1/c$ under the condition that $\Delta (r)>0$.
The  maximal interval $I$ is determined by the conditions that  $\Delta (r)>0$ and $r^2>1/c$. Noting that  the discriminant of $c^2r^4-2\delta r^2+1$ is $\delta ^2-c^2>0$ and by the following inequalities 
\begin{equation*}
\cfrac{\delta -\sqrt{\delta ^2-c^2}}{c^2}<\cfrac{1}{c}<\cfrac{\delta +\sqrt{\delta ^2-c^2}}{c^2},
\end{equation*}
we can take $I$ which satisfies $r^2>(\delta+\sqrt{\delta ^2-c^2})/c^2$. The proof of (ii) has been completed.\\

Third, we determine the condition that a Type II surface $X$ has both spacelike part and lightlike part.  By (\ref{eq:outerproduct}), this condition is equivalent to
\[
\begin{cases}{}
\phi \psi \leq 0\text{, and} \\
\phi \text{ or }\psi \text{ have a zero point.}\nonumber
\end{cases}
\]
Let $r \in I$ be fixed arbitrarily. If there exists a $s>0$ such that $\phi (r, s)<0$, then $\phi \psi (r, s)>0$, which is a contradiction. Therefore we get $\phi \geq 0$, especially $a-b \geq 0$. Similarly, if there exists a $s>0$ such that $\psi (r, s)>0$, then $\phi \psi (r, s)>0$, which is a contradiction. Therefore we get $a-b \leq 0$, and then $a=b$. In this case, $\phi$ and $\psi$ are written as (\ref{eq:phia=b}) and (\ref{eq:psia=b}).

Since $\phi \geq 0$ and $\psi \leq 0$, we get $a<0$ and $a>0$, respectively. This is a contradiction, and hence, the proof of (iii) has been completed.\\

Fourth, we determine the condition that a Type II surface $X$ has both a timelike part and a lightlike part.  By (\ref{eq:outerproduct}) and the relation $\phi =\psi +4s$, the condition is equivalent to
\begin{subnumcases}{}
\phi \leq 0\text{ on }I\times \mathbb{R}_{>0}\text{ and }\phi \ \text{has a zero point, or}\label{eq:phi<} \\
\psi \geq 0\text{ on }I\times \mathbb{R}_{>0}\text{ and }\psi \text{ has a zero point.}\label{eq:psi>}
\end{subnumcases}

 Let us consider the condition (\ref{eq:phi<}).
Now, we get $a-b<0$ from the condition, and there exists a $r\in I$ such that
\begin{subnumcases}{}
(-\Delta (r)+1)^2+(a^2-b^2)r^4=0,\ \text{and} \label{eq:determinant} \\
s_1(r):=\cfrac{\Delta (r)-1}{(a-b)r^2}>0. \label{eq:s_1>0}
\end{subnumcases}
Since $a<b$, (\ref{eq:s_1>0}) is equivalent to $\Delta (r)-1<0$. Therefore by (\ref{eq:determinant}), we get $a^2-b^2<0$ and take $c:=\sqrt{-a^2+b^2}>0$. Using the constant $c$ and the condition  $\Delta (r)-1<0$, (\ref{eq:determinant}) is equivalent to $-\Delta (r)+1-cr^2=0$, that is, $c=\delta$, and in this case we get the $r$ satisfies the condition that $r^2<1/c$. Conversely, if we assume that 
$a-b<0$, $a^2-b^2<0$, $c=\delta $, and $r^2<1/c$, then we get $\Delta (r)=1-cr^2$, (\ref{eq:determinant}), and (\ref{eq:s_1>0}) for any $r \in I$, and hence, (\ref{eq:phi<}) also holds. In this case we can take $r_0=0$ in the equation (\ref{eq:type2}).

Next, we determine the lightlike part of the surface. By (\ref{eq:s_1>0}) and $\Delta(r)=1-cr^2$, the surface is lightlike on the domain
\begin{equation}
s_1(r)=s_1=\cfrac{c}{b-a}. \nonumber
\end{equation}
If we define $\theta_1:=\log s_1$, the lightlike part of the surface is
\begin{equation}
c(r):=X(r,\theta _1) =\left(
\displaystyle\int^r _0 \cfrac{1}{\Delta (s)}\,ds, \int^r_0 \cfrac{as^2}{\Delta (s)}\,ds, \int^r _0 \cfrac{bs^2}{\Delta (s)}\,ds \right) +\left( 0, r\sinh\theta _1, r\cosh\theta _1
 \right) ,\nonumber
\end{equation}
Since $\sinh\theta _1=\cfrac{a}{c}$ and $\cosh\theta _1 =\cfrac{b}{c}$, 
\begin{equation}
c(r) =
\cfrac{1}{c}\displaystyle\int^r _0 \cfrac{1}{\Delta (s)}\,ds\left( c, a, b \right ). \nonumber
\end{equation}
Therefore the lightlike part is a part of a straight line.

Next, let us consider the condition(\ref{eq:psi>}). Now, we get $a-b>0$ from the condition, and there is a $r \in I$ such that 
\begin{subnumcases}{}
(\Delta (r)+1)^2+(a^2-b^2)r^4=0\text{ and}    \label{eq:determinant_2} \\
s_2:=\cfrac{\Delta (r)+1}{(a-b)r^2}>0. \label{eq:s_2>0}
\end{subnumcases}
Since $a-b>0$, (\ref{eq:s_2>0}) always holds. By (\ref{eq:determinant_2}), we get $a^2-b^2<0$ and take $c:=\sqrt{-a^2+b^2}>0$, (\ref{eq:determinant_2}) is equivalent to $\Delta (r)+1-cr^2=0$. By squaring both sides, we get $c=\delta $ and $r^2>1/c$. Conversely, if we assume that 
$a-b>0$, $a^2-b^2<0$, $c=\delta $, and $r^2>1/c$, then we get $\Delta (r)=cr^2-1$, (\ref{eq:determinant_2}), and (\ref{eq:s_2>0}) for any $r \in I$, and hence, (\ref{eq:psi>}) also holds. \\
Next, we determine the lightlike part of the surface. By (\ref{eq:s_2>0}) and $\Delta(r)=cr^2-1$, the surface is lightlike on the domain
\begin{equation}
s_2(r)=s_2=\cfrac{c}{a-b}. \nonumber
\end{equation}
If we define $\theta_2:=\log s_2$, the lightlike part of the surface is
\begin{equation}
c(r):=X(r,\theta _2) =\left(
\displaystyle\int^r _{r_0} \cfrac{1}{\Delta (s)}\,ds, \int^r_{r_0} \cfrac{as^2}{\Delta (s)}\,ds, \int^r _{r_0} \cfrac{bs^2}{\Delta (s)}\,ds \right) +\left( 0, r\sinh\theta _2, r\cosh\theta _2
 \right) . \nonumber
\end{equation}
Since $\sinh\theta _2=-\cfrac{a}{c}$ and $\cosh\theta _2=-\cfrac{b}{c}$, 
\begin{equation}
c(r) =
\cfrac{1}{c}\displaystyle\int^r _{r_0} \cfrac{1}{\Delta (s)}\,ds\left( c, a, b \right ) -\cfrac{1}{c}\left( 0, a, b \right )r_0. \nonumber
\end{equation}
Therefore, the lightlike part is a part of a straight line, and hence, the proof of (iv) has been completed.\\

 Fifth, we determine the condition that a Type II surface $X$ has all causal characters. The surface $X$ is lightlike on $I\times \mathbb{R}$ if and only if $\phi$ and $\psi$ are zero functions. In this case, we get $a=b=0$, which is a contradiction. Therefore, the Type II surfaces except those stated above are ZMC surfaces which have all causal characters.
\newline

Now we prove (2). For the Type I surface $X$, by a straightforward computation, the determinant of the induced metric of the surface (\ref{eq:Singular_hyperbpla1}) is $-(1/c^3)\left( a\sinh\theta -b\cosh\theta \right)^2$. If we assume that $a\sinh\theta -b\cosh\theta =0$, then we get the equation $e^{2\theta}=(a+b)/(a-b)$, which contradicts the condition $-a^2+b^2>0$. Therefore, the surface is timelike. 
 
 For the Type II surface $X$, by a straightforward computation, the determinant of the induced metric of the surface (\ref{eq:Singular_hyperbpla2}) with respect to the local coordinate system $(u, \theta)$ is $-(1/c^3)\left( b\sinh\theta -a\cosh\theta \right)^2$. Therefore, if we define $\theta_0:=\frac{1}{2}\log{(a+b)/(-a+b)}$, the surface has timelike part and the following lightlike line:
\begin{equation}
 c(u):=X(u, \theta_0)=\frac{u}{c} \left( c, a, b \right )+\cfrac{1}{\sqrt{c}}\left( 0, \sinh \theta_0, \cosh \theta_0 \right) .\nonumber \\
\end{equation}
\end{proof}

At the end of this subsection, it should be noted that there is a relationship between singular type surfaces and the asymptotic behaviors of some general type surfaces. If we consider a Type II ZMC surface of Riemann type which belongs to the class (iv) in Theorem \ref{thm:causaloftimelike}, this surface has the following implicit form up to a translation:
\begin{equation*}
\left(t-\cfrac{b}{c}x+\cfrac{b}{c}r(x)\right)^2-\left(y-\cfrac{a}{c}x+\cfrac{a}{c}r(x)\right)^2=r(x)^2, r(t)=\cfrac{1}{\sqrt{c}}\tanh{(\sqrt{c}x)}.
\end{equation*}
By taking limits $r\to \pm 1/\sqrt{c}$ on the surface (which correspond to $x\to \pm \infty$), we obtain 
\begin{equation*}
\left(t-\cfrac{b}{c}x\pm\cfrac{b}{c\sqrt{c}}\right)^2-\left(y-\cfrac{a}{c}x\pm\cfrac{a}{c\sqrt{c}}\right)^2=\cfrac{1}{c}.
\end{equation*}
These surfaces are nothing but the singular type surfaces (\ref{eq:Singular_hyperbpla2}) up to a translation in $\mathbb{L}^3$ and Figure 7 and 8 in Section 5.3 are pictures of the above two surfaces.

 \subsection{Causal characters of ZMC surfaces of Riemann type foliated by parabolas\label{parabola}}ZMC surfaces of Riemann type which are foliated by circles in lightlike parallel planes are written as in Theorem \ref{thm:lightlike}.
\begin{theorem}\label{thm:causaloflightlike}
Let $J$ be the maximal interval of the functions $r$, $f$, and $g$ in Theorem \ref{thm:lightlike}. Causal characters of ZMC surfaces of Riemann type which are foliated by circles in lightlike parallel planes and satisfy one of the four conditions in Theorem \ref{thm:lightlike} are determined as follows.
\begin{itemize}
\item[(1)] General type: Causal characters of general type surfaces are determined as follows.
\begin{itemize}
\item[(i)] The surface $X$ is spacelike on $J\times \mathbb{R}$ if and only if $a>0$ and $p<0$. \item[(ii)] The surface $X$ is timelike on $J\times \mathbb{R}$ if and only if $a<0$ and $p<0$.\item[(iii)] The surface $X$ has both a spacelike part and a lightlike part and it does not have a timelike part if and only if $a>0$ and $p=0$. Moreover in this case, the lightlike part is a part of a straight line.
\item[(iv)] The surface $X$ has both a timelike part and a lightlike part and it does not have a spacelike part if and only if $a<0$ and $p=0$. Moreover in this case, the lightlike part is a part of a straight line. 
\item[(v)] The surfaces except those stated above are ZMC surfaces which change their causal characters.
\end{itemize}
\item[(2)] Singular type: Causal characters of singular type surfaces are determined as follows.
\begin{itemize}
\item[(i)] The surface $X$ is timelike if and only if $p>0$.
\item[(ii)] The surface $X$ has both a timelike part and a lightlike part and it does not have a spacelike part if and only if $p=0$. Moreover in this case, the lightlike part is a straight line.
\item[(iii)] The surface $X$ has all causal characters if and only if $p<0$. 
\end{itemize}
\end{itemize}
\end{theorem}
\begin{proof}
First we prove (1). By (\ref{eq:4}), the determinant of the induced metric with respect to the local coordinate system $(u, v)$ is $EG-F^2=4g'(u)-4r(u)f'(u)v+(2r'(u)-4r(u)^2)v^2$. By substituting the equations (\ref{eq:r}) and (\ref{eq:f}), we get
\begin{equation}
EG-F^2=2\{ 2g'(u)-2\cfrac{b}{r(u)}v+av^2 \} . \label{eq:causal} \\
\end{equation}
Then causal characters are determined as follows.

First, we determine the condition that a surface X is spacelike. If we assume $a=b=0$, then it contradicts Theorem \ref{thm:lightlike}. Therefore by considering the discriminant of the right-hand side of (\ref{eq:causal}) for every $u \in J$, the spacelike condition on $J \times \mathbb{R}$ is $a>0$ and for arbitrary $u\in J$, $(b^2/r(u)^2)-2ag'(u)<0$.
When $a>0$, by Theorem \ref{thm:lightlike} we get $(b^2/r(u)^2)-2ag'(u)=-2ap\cos ^2 \{ \sqrt{2a}(u+c) \}$. Therefore $X$ is spacelike if and only if $a>0$ and $p>0$. The maximal interval $J$ is determined by the condition that $r$, $f$, and $g$ are well-defined, that is, $\cos \{ \sqrt{2a}(u+c)\} \neq 0$, and $\sin \{ \sqrt{2a}(u+c)\}\neq 0$. The proof of (i) has been completed.

 Second, we determine the condition that a surface X is timelike. The timelike condition can be determined the same as the case (i). By (\ref{eq:causal}), the surface is timelike on $J \times \mathbb{R}$ if and only if $a<0$ and $(b^2/r(u)^2)-2ag'(u)<0$ on $J$. When $a<0$, by Theorem \ref{thm:lightlike}, we get $(b^2/r(u)^2)-2ag'(u)=-2ap\cosh ^2 \{ \sqrt{-2a}(-u+c) \}$. Therefore $X$ is timelike if and only if $a<0$ and $p<0$. The  maximal interval $J$ is determined by the condition that $r$, $f$, and $g$ are well-defined, that is, $J\subset \mathbb{R} \setminus\{c \}$. The proof of (ii) has been completed.

Next we will check the condition that a surface has two causal characters. Note that If $a=0$, the causal character of the surface changes from spacelike to timelike along the curve $v=(r(u)/b)g'(u)$. Third, we consider the condition that the surface has both spacelike part and lightlike part. Since the surface does not have the timelike part, we get $a>0$. Considering the discriminant of the right-hand side of (\ref{eq:causal}) for every $u \in J$, the condition is equivalent to $a>0$ and for arbitrary $u\in J$, $(b^2/r(u)^2)-2ag'(u)\leq 0$ and there exists a $u\in J$ such that $(b^2/r(u)^2)-2ag'(u)= 0$.

Moreover, since $(b^2/r(u)^2)-2ag'(u)= -2ap\cos ^2 \{ \sqrt{2a}(u+c) \}$, the above condition is equivalent to $a>0$ and $p=0$. Next we determine the shape of the lightlike part. By (\ref{eq:causal}),  the surface is lightlike on the domain 
\begin{equation}\label{eq:v}
v(u)=\cfrac{b}{ar(u)}=\cfrac{\sqrt{2}b}{a\sqrt{a}}\cot \{ \sqrt{2a} (u+c) \},\quad u \in J, 
\end{equation}
and the lightlike part is 
\begin{equation}
 c(u)=X(u,v(u)) =\left( f(u)+v(u), g(u)+u+\cfrac{r(u)v(u)^2}{2}, g(u)-u+\cfrac{r(u)v(u)^2}{2}\right ). \nonumber
\end{equation}
Since $a>0$ and $p=0$, we get 
\begin{eqnarray}
f(u)&=&-\cfrac{\sqrt{2}b}{a^{\frac{3}{2}}}\cot\{ \sqrt{2a} (u+c) \}-\cfrac{2b}{a}(u+c) \nonumber \\ 
&=&-v(u)-\cfrac{2b}{a}(u+c). \nonumber
\end{eqnarray}
By (\ref{eq:v}), $r(u)v(u)^2/2=(b/2a)v(u)$ and then 
\begin{equation}
g(u)+\cfrac{r(u)v(u)^2}{2}=-\cfrac{b^2}{a^2}(u+c).\nonumber
\end{equation}
Therefore the lightlike part is the following line:
\begin{equation}
 c(u)=u \left( -\cfrac{2b}{a}, 1-\cfrac{b^2}{a^2}, -1-\cfrac{b^2}{a^2} \right )+\left( -\cfrac{2bc}{a}, -\cfrac{b^2c}{a^2}, -\cfrac{b^2c}{a^2} \right).\nonumber
\end{equation}
The proof of (iii) has been completed.

Fourth, we consider the condition that a surface has both a timelike part and a lightlike part. Since the surface does not have the spacelike part, we get $a<0$. Considering the discriminant of the right-hand side of (\ref{eq:causal}) for every $u \in J$, the condition is equivalent to $a<0$ and for arbitrary $u\in J$, $b^2/r(u)^2-2ag'(u)\leq 0$ and there exists a $u\in J$ such that $b^2/r(u)^2-2ag'(u)= 0$.
Moreover, since $b^2/r(u)^2-2ag'(u)= -2ap\cosh ^2 \{ \sqrt{-2a}(-u+c) \}$, the above condition is equivalent to $a<0$ and $p=0$.\\
Next we determine the shape of the lightlike part. By (\ref{eq:causal}),  the surface is lightlike on the domain 
\begin{equation}\label{eq:v'}
v(u)=\cfrac{b}{ar(u)}=\cfrac{\sqrt{2}b}{a\sqrt{-a}}\coth \{ \sqrt{-2a} (-u+c) \} , u \in J, 
\end{equation}
and the lightlike part is 
\begin{equation}
 c(u)=X(u,v(u)) =\left( f(u)+v(u), g(u)+u+\cfrac{r(u)v(u)^2}{2}, g(u)-u+\cfrac{r(u)v(u)^2}{2}\right ). \nonumber
\end{equation}
Since $a<0$ and $p=0$, we get 
\begin{eqnarray}
f(u)&=&-\cfrac{\sqrt{2}b}{a\sqrt{-a}}\coth\{ \sqrt{-2a} (-u+c) \}-\cfrac{2b}{a}(u+c)  \nonumber \\ 
&=&-v(u)-\cfrac{2b}{a}(u+c). \nonumber
\end{eqnarray}
By (\ref{eq:v}), $r(u)v(u)^2/2=(b/2a)v(u)$ and then 
\begin{equation}
g(u)+\cfrac{r(u)v(u)^2}{2}=\cfrac{b^2}{a^2}(-u+c).\nonumber
\end{equation}
Therefore the lightlike part is the following line:
\begin{equation}
 c(u)=u \left( -\cfrac{2b}{a}, 1-\cfrac{b^2}{a^2}, -1-\cfrac{b^2}{a^2} \right )+\left( -\cfrac{2bc}{a}, -\cfrac{b^2c}{a^2}, -\cfrac{b^2c}{a^2} \right).\nonumber
\end{equation}
The proof of (iv) has been completed.

  Fifth, we determine the condition that a surface $X$ has all causal characters. Since $(a, b)\neq(0, 0)$, $X$ is not to be lightlike. Therefore the surface which does not satisfy the condition from (i) to (iv) is a surface which changes its causal characters. The proof of (v) and so the proof of (1) have been completed.
\newline

Now we prove (2). The discriminant of the right-hand side of (\ref{eq:causal}) is determined by $b^2/r(u)^2-2ag'(u)=4a\sqrt{-2a}pe^{-2\sqrt{-2a}u}$. Therefore causal characters of the surface $X$ is determined by only $p$.

When $p>0$, the discriminant of the right-hand side of (\ref{eq:causal}) is negative. Therefore the surface $X$ is timelike. The proof of (i) has been completed. 

When $p=0$, the discriminant of the right-hand side of (\ref{eq:causal}) is identically zero. Therefore the surface $X$ has a timelike part and a lightlike part. Since 
\begin{equation*}
EG-F^2=-4\left( \cfrac{b}{a}-\sqrt{\cfrac{-a}{2}}v\right)^2, 
\end{equation*}
the surface $X$ is lightlike on $v=v_0:=(b/a)\sqrt{2/-a}$, therefore the lightlike part of the surface is the following straight line:
\begin{equation*}
c(u)=X(u, v_0)=\left( -\cfrac{2b}{a}+v_0, -\frac{b^2}{a^2}u+u+\frac{1}{2}\sqrt{\frac{-a}{2}}v_0^2, -\frac{b^2}{a^2}u-u+\frac{1}{2}\sqrt{\frac{-a}{2}}v_0^2\right).
\end{equation*}
The proof of (ii) has been completed. 

 When $p<0$, the discriminant of the right-hand side of (\ref{eq:causal}) is positive. Therefore the surface $X$ has all causal characters. The proof of (iii) and so the proof of (2) have been completed. 
\end{proof}
Similarly, as Section 4.1 and 4.2, it should be noted that there is a relationship between singular type surfaces and asymptotic behaviors of some general type surfaces. If we consider a general type surface which belongs to the class (1) (iv) in Theorem \ref{thm:causaloflightlike}, this surface has the parametrization (\ref{eq:4}) with the following functions:
\begin{center}
$r(u)=\sqrt{\cfrac{-a}{2}}\tanh\{ \sqrt{-2a} (-u+c) \}$, \\
$f(u)=-\cfrac{b}{ar(u)}-\cfrac{2b}{a}(u+c)$, \\
$g(u)=-2\cfrac{b^2}{a^2r(u)}+\cfrac{b^2}{a^2}(-u+c)$. \\
\end{center}
By taking a limit $r\to \sqrt{-a/2}$, we obtain the functions $r$, $f$, and $g$ of the singular type surface which belongs to the class (2) (ii) in Theorem \ref{thm:causaloflightlike}. Therefore, the corresponding surface is a singular type surface which has timelike part and a lightlike line. Figure 10 and 11 in Section 5.4 are pictures of the above two surfaces.\\

By Theorems \ref{thm:causalofspacelike}, \ref{thm:causaloftimelike}, and \ref{thm:causaloflightlike}, we classified ZMC surfaces of Riemann type according to their causal characters. In particular, we obtain the following corollary for ZMC surfaces of Riemann type which have exactly two causal characters, that is, only spacelike and lightlike parts, or only timelike and lightlike parts.
\begin{corollary}\label{thm:2-causal}
If a ZMC surface of Riemann type has exactly two causal characters, then the lightlike part is a part of a straight line.
\end{corollary}
The following figures are Type II ZMC surfaces of Riemann type (\ref{eq:type2}) whose parameter $\delta$ varies from $\sqrt{2}$ to $1/2$. In the second figure, the surface has timelike part and a lightlike line.
\begin{figure}[htbp]
\begin{center}
\hspace{-1.5cm} 
\begin{tabular}{c}
\begin{minipage}{0.3\hsize}
\begin{center}
\hspace{-0.6cm} 
\includegraphics[clip,scale=0.29,bb=0 0 440 350]{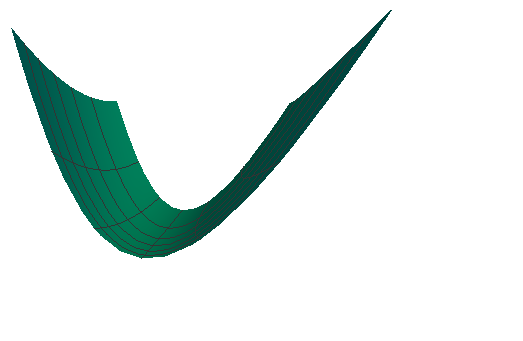}
 \hspace{1.6cm} $\delta =\sqrt{2}$: Timelike
\end{center}
\end{minipage}

\begin{minipage}{0.3\hsize}
\begin{center}
 \vspace{-3cm} 
\includegraphics[clip,scale=0.315,bb=0 0 540 530]{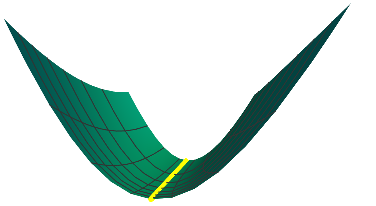}%
 \vspace{1cm} 
 $\delta =1$: Lightlike and timelike
\end{center}
\end{minipage}

\begin{minipage}{0.3\hsize}
\begin{center}
\vspace{-3cm}
\includegraphics[clip,scale=0.3,bb=0 0 540 530]{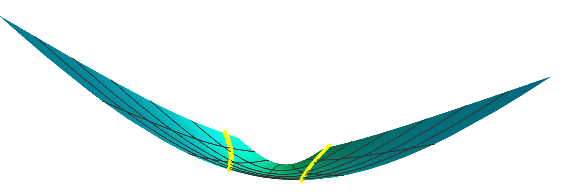}%
 \vspace{1cm} 
 $\delta =\frac{1}{2}$: All causal characters
\end{center}
\end{minipage}
\end{tabular}
\caption{Type II ZMC surfaces of Riemann type (\ref{eq:type2}) whose parameter $\delta$ vary from $\sqrt{2}$ to $1/2$.}
\end{center}
\end{figure}

\begin{remark}
It is not necessary that a ZMC surface of Riemann type which has a lightlike line (or a part of a straight line) has exactly two causal characters. In fact, if we consider the following special case of the case (v) in Theorem \ref{thm:causaloftimelike}: $b=-2|a|$ and $r^2>1/c$, each surface in this class has all causal characters and a lightlike half-line. Moreover, the lightlike part of the surface in this class consists of a lightlike half-line and two non-degenerate null curves, and the surface is timelike along the lightlike half-line and changes its causal characters along non-degenerate null curves. We can extend this surface along the straight line which is regarded as a circle with infinite radius, and the extended surface has a lightlike line. Figure 2 shows a surface in this class.
\end{remark}
\begin{figure}[htbp]
\begin{center}
\includegraphics[clip,scale=0.40,bb=0 0 410 350]{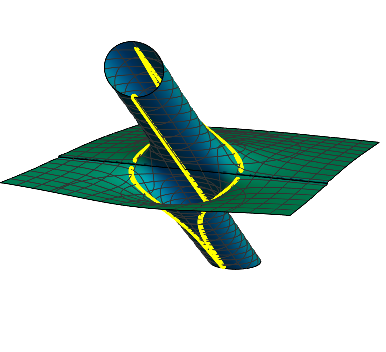}%
 \vspace{-1cm}
\caption{A surface which has all causal characters and whose lightlike part consists of a straight line and two non-degenerate null curves.}
\end{center}
\end{figure}


\section{ZMC surfaces containing a lightlike line\label{withlightlikeline}}
\subsection{Characteristic of ZMC surfaces containing a lightlike line}In this section, we calculate characteristics which were introduced in \cite{FujimoriETAL1} and determine the types of ZMC surfaces of Riemann type containing a lightlike line which appeared in the previous section. First, we review the definition of the characteristic of a ZMC surface containing a lightlike line, and give a useful lemma to calculate characteristics. 
 
Suppose that a ZMC surface with parameters $(u, v)$:
\begin{center}
$X(u, v)=(x(u, v), y(u, v), t(u, v))$,
\end{center}
has a lightlike part $L$ contained in  $\{ (0, t, t) \mid t \in \mathbb{R} \}$. Near the lightlike line segment $L$, the surface can be written as a graph on the $xy$-plane:
\begin{center}
$t=t(x, y)=t(x(u, v), y(u,v))$.
\end{center}
By the condition that mean curvature is identically zero, if we set $\alpha (y)=t_{xx}(0, y)$, then $\alpha $ satisfies the following ODE:
\begin{equation}
\frac{d\alpha}{dy}(y)+\alpha ^2(y)+\mu =0, \label{eq:b_2}
\end{equation}
where $\mu $ is a real constant called the {\it characteristic} of the ZMC surface along $L$. The notion of the characteristic of a ZMC surface was introduced in \cite{FujimoriETAL1}. There several examples of ZMC surfaces containing a lightlike line were constructed, and their  $\alpha$ and $\mu$ were determined (see Table 2). By a homothetic change, the characteristic $\mu$ can be $-1$, $0$, or $1$, and there is a relationship between the sign of the characteristic of a surface and its causal characters as proved in \cite{FujimoriETAL1}.
\begin{proposition}[{\cite[Proposition 1]{FujimoriETAL1}}]
If $\mu>0$, the surface is spacelike on both sides of $L$, and if $\mu<0$, the surface is timelike on both sides of $L$.
\end{proposition}
If $\mu=0$, we can not determine causal characters across the lightlike line; however, surfaces which change their causal characters from spacelike to timelike across a lightlike line have always characteristic $\mu=0$.

Table 1 shows the explicit solutions $\alpha (y)$ of (\ref{eq:b_2}), and  Table 2 shows examples of ZMC surfaces containing a lightlike line. See the references \cite{FujimoriETAL1,FujimoriETAL2,FujimoriETAL3} for details.\\

The following lemma is useful to calculate the characteristics of ZMC surfaces.
\begin{lemma}\label{lemma:characteristic}
In the above situation, it holds that on the lightlike line
 \begin{equation*}
 t_{uu}(0,y)=\alpha (y)x_u(0,y)^2+y_{uu}(0,y).
 \end{equation*}
\end{lemma}
\begin{proof}
By the chain rule,
\begin{equation}
t_{uu}=t_{xx}x_u^2+2t_{xy}x_uy_u+t_xx_{uu}+t_{yy}y_u^2+t_yy_{uu}. \nonumber
\end{equation}
By the assumption that the lightlike line $L$ is in $\{ (0, t, t) \mid t \in \mathbb{R} \}$, we get $t_x(0,y)=0$, $t_y(0,y)=1$, $t_{xy}(0,y)=0$, and $t_{yy}(0,y)=0$. 
Substitute these partial derivatives into the above equality, then we obtain the formula.
\end{proof}
 \begin{table}[ht]\label{pic1}
\begin{center}
    \begin{tabular}{|l|c|r|r|} \hline 
      \multicolumn{1}{|c|}{$\mu $} & \multicolumn{3}{|c|}{Types of $\alpha $} \\  \hline\hline
      \multicolumn{1}{|c|}{$ \mu =1$ } & \multicolumn{3}{|c|}{$ \alpha^+:=-\tan(y+c)$} \\ \hline
     \multicolumn{1}{|c|}{$ \mu =0$ } & \multicolumn{1}{|c|}{$ \alpha^0_{\rm I}:=0$} & \multicolumn{2}{|c|}{$ \alpha^0_{\rm II}:=(y+c)^{-1}$} \\ \hline
     \multicolumn{1}{|c|}{$ \mu =-1$} & \multicolumn{1}{|c|}{$ \alpha^{-}_{\rm I}:=\tanh(y+c)$} & \multicolumn{1}{|c|}{$ \alpha^{-}_{\rm II}:=\coth(y+c)$} & \multicolumn{1}{|c|}{$ \alpha^{-}_{\rm III}:=\pm 1$} \\ \hline
    \end{tabular}
          \caption{$\mu $ and $\alpha $.}
  \end{center}
  \end{table}
  
  \begin{table}[ht]
  \begin{center}
    \begin{tabular}{|l|c|r|r|} \hline 
      \multicolumn{1}{|c|}{$\alpha $} & \multicolumn{3}{|c|}{Examples of ZMC surfaces} \\  \hline\hline
      \multicolumn{1}{|c|}{$ \alpha^+$} & \multicolumn{3}{|c|}{Spacelike Scherk surface} \\ \hline
     \multicolumn{1}{|c|}{$ \alpha^0_{\rm I}$} & \multicolumn{3}{|c|}{The lightlike plane, surfaces which change type constructed in \cite{FujimoriETAL2}}  \\ \hline
     \multicolumn{1}{|c|}{$ \alpha^0_{\rm II}$} & \multicolumn{3}{|c|}{Parabolic and hyperbolic catenoids, the light cone}  \\ \hline
      \multicolumn{1}{|c|}{$ \alpha^{-}_{\rm I}$} & \multicolumn{3}{|c|}{Timelike Scherk surface of second kind}  \\ \hline
     \multicolumn{1}{|c|}{$ \alpha^-_{\rm II}$} & \multicolumn{3}{|c|}{Timelike Scherk surface of first kind}  \\ \hline
      \multicolumn{1}{|c|}{$ \alpha^-_{\rm III}$} & \multicolumn{3}{|c|}{A Klyachin's example in \cite{Klyachin}}  \\ \hline
    \end{tabular}
    \caption{Examples of ZMC surfaces according to types of $\alpha$.}
\end{center}
  \end{table}
 \subsection{ZMC surfaces of Riemann type foliated by Euclidean circles containing a lightlike line}
By Theorem \ref{circle}, if a ZMC surface of Riemann type which is foliated by circles in spacelike parallel planes either has exactly two causal characters or satisfies the conditions that $b=-2a$ and $r^2>1/a$, the surface contains a lightlike line segment. Each of ZMC surfaces (\ref{eq:splike}) contains a lightlike line (segment), and each of the ZMC surfaces (\ref{eq:Singular_splike}) contains two lightlike lines.

\begin{theorem}\label{thm:sp}
ZMC surfaces of Riemann type which are given by (\ref{eq:splike}) or (\ref{eq:Singular_splike}) and have lightlike lines have following characteristics of ZMC surfaces.
\begin{itemize}
\item[(1)] General type:
\begin{enumerate}
  \item[(i)] In the case when $b=2a$, the surface
\begin{center}
 $X(r,\theta) = \left(\displaystyle\int ^r _{0}\cfrac{as^2}{as^2+1}~ds+r\cos\theta , r\sin\theta , \displaystyle\int ^r _{0}\cfrac{1}{as^2+1}~ds
\right)$ \\[8pt]
\end{center}
is an $\alpha ^+$-type ZMC surface with the characteristic  $\mu=|a|$ along its lightlike line (see Figure 3).
\item[(ii)] In the case when $b=-2a$ and $0<r^2<1/a$, the surface
\begin{center}
 $X(r,\theta) = \left(\displaystyle\int ^r _{0}\cfrac{as^2}{1-as^2}\,ds+r\cos\theta , r\sin\theta , \displaystyle\int ^r _{0}\cfrac{1}{1-as^2}\,ds
\right)$ \\[8pt]
\end{center}
is an $\alpha ^-_{\rm II}$-type ZMC surface with the characteristic  $\mu=-a$ along its lightlike line (see Figure 4).
\item[(iii)] In the case when $b=-2a$ and $r^2$, $r_0^2>1/a$, the surface
\begin{center}
 $X(r,\theta) = \left(\displaystyle\int ^r _{r_0}\cfrac{as^2}{as^2-1}\,ds+r\cos\theta , r\sin\theta , \displaystyle\int ^r _{r_0}\cfrac{1}{as^2-1}\,ds
\right)$ \\[8pt]
\end{center}
is an $\alpha ^-_{\rm I}$-type ZMC surface with the characteristic  $\mu=-a$ along its lightlike line (see Figure 2).
\end{enumerate} 
\item[(2)] Singular type:
every singular type surface (\ref{eq:Singular_splike}) is an $\alpha ^-_{\rm III}$-type ZMC surface with the characteristic  $\mu=-a$ along each of its lightlike lines (see Figure 5).
\end{itemize}
\end{theorem}

\begin{proof}
First we prove (1). Let us consider a surface $X$ which satisfies the condition (i). By the proof of Theorem \ref{thm:causalofspacelike}, the lightlike line of the surface is 
\begin{equation*}
c(r)=X(r,\pi) =\displaystyle\int^r_0 \cfrac{1}{1+as^2}~ds \left( -1, 0, 1 \right).
\end{equation*}
By considering the following coordinate transformation on $\mathbb{L}^3$,
\begin{center}
$  A= \left(
    \begin{array}{ccc}
      0 & 1 & 0 \\
      -1 & 0 & 0 \\
      0 & 0 & 1 
 \end{array}
  \right) $,
  \end{center}
and the immersion $\tilde{X}(r, \theta) := A\circ X(r, \theta)$. Then $\tilde{X}$ is written as
\begin{align}
 &\tilde{X}(r, \theta) =\left( x(r, \theta),  y(r, \theta),  t(r, \theta) \right) \nonumber \\
&=\left(  r\sin\theta  , -\displaystyle\int ^r _{0}\cfrac{as^2}{as^2+1}\,ds-r\cos\theta , \displaystyle\int ^r _{0}\cfrac{1}{as^2+1}\,ds \right) , \nonumber
\end{align}
and the lightlike part of the surface is in $\{ (0, t, t) \mid t \in \mathbb{R} \}$. We write $\tilde{X}$ as $X$ again and determine the function $\alpha$.

By computing partial derivatives of coordinate functions on the lightlike line, $x_{\theta}(r, \pi)=-r$, $y_{\theta \theta}(r, \pi)=-r$, and $t_{\theta \theta}(r, \pi )=0$. By using Lemma \ref{lemma:characteristic}, we get 
\begin{equation*}
0=\alpha (y(r))r^2-r.
\end{equation*}
Therefore we obtain $\alpha (y(r))=1/r$. \\

On the other hand, by differentiating $y(r, \theta)$ with respect to $y$,
\begin{equation}
 1=-\frac{ar^2}{ar^2+1}r_y -r_y\cos \theta +r\sin \theta \theta _y,\nonumber
\end{equation}
and then we get $r_{y}(r, \pi )=ar^2+1$, and hence,
\begin{align}
\frac{d\alpha}{dy}(y(r))+\alpha ^2(y(r))&=-\frac{1}{r^2}\times (ar^2+1) +\frac{1}{r^2}  \nonumber \\
&=-a. \nonumber
\end{align}
Therefore we obtain the characteristic $\mu=a$, and hence the surface is an $\alpha ^+$-type ZMC surface with the characteristic $\mu =a$.

 For any surface $X$ which satisfies the condition (ii), we can check that $\alpha (y(r))=1/r$ and $\mu=-a$ by a similar calculation as in the case (i). Since $r^2<1/a$, the function $\alpha$ is unbounded. Therefore the surface is an $\alpha _{\rm II}^-$-type ZMC surface with the characteristic $\mu =-a$. 
 
For any surface $X$ which satisfies the condition (iii), we can check that $\alpha (y(r))=-1/r$ and $\mu=-a$ by a similar calculation as in the case (i) and (ii). Since $r^2>1/a$, the function $\alpha$ is bounded. Therefore the surface is an $\alpha _{\rm I}^-$-type ZMC surface with the characteristic $\mu =-a$. \\

Now we prove (2). We can check that $\alpha (y(r))=-\sqrt{a}$ and $\mu=-a$ along each of its lightlike lines by a similar calculation as in the above cases. Since the function $\alpha$ is a constant function, the surface is an $\alpha _{\rm III}^-$-type ZMC surface with the characteristic $\mu =-a$ along each of its lightlike lines. The proof has been completed.
\end{proof}
\begin{figure}[!h]
\begin{center}
\begin{tabular}{c}
\hspace{-1.2cm}
\begin{minipage}{0.4\hsize}
\begin{center}
\vspace{-0cm}
\includegraphics[clip,scale=0.30,bb=0 0 500 409]{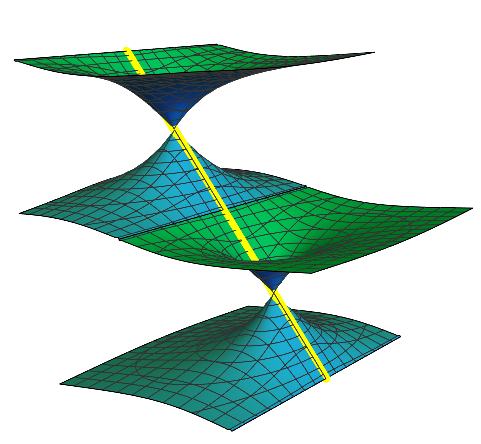}%
\vspace{0.5cm}
\caption{An $\alpha^+$-type surface foliated by Euclidean circles and straight lines (general type).}
\end{center}
\end{minipage}
\hspace{-1.0cm}
\begin{minipage}{0.4\hsize}
\begin{center}
\vspace{-0.4cm}
\includegraphics[clip,scale=0.30,bb=0 0 355 459]{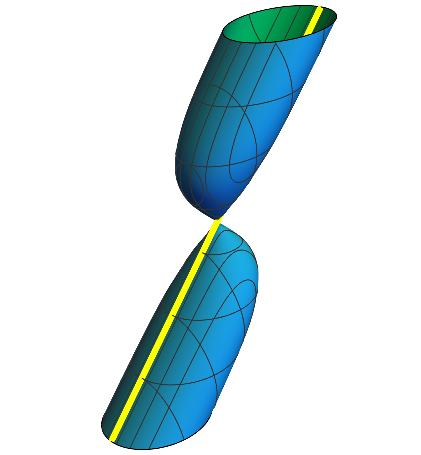}%
\vspace{0cm}
\hspace{-1cm}
\caption{An $\alpha^-_{\rm II}$-type surface foliated by Euclidean circles (general type).}
\end{center}
\end{minipage}
\hspace{-0.5cm}
\begin{minipage}{0.4\hsize}
\begin{center}
\vspace{-0.5cm}
\includegraphics[clip,scale=0.30,bb=0 0 455 449]{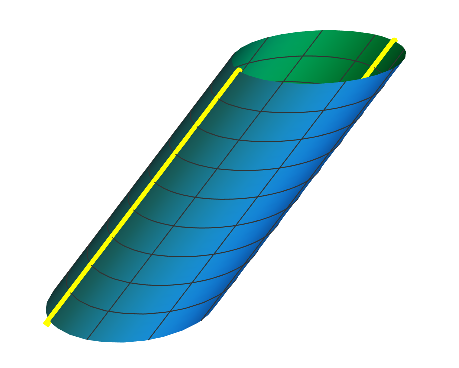}%
\vspace{0.3cm}
\hspace{-0.3cm}
\caption{An $\alpha^-_{\rm III}$-type surface foliated by Euclidean circles (Singular type).}
\end{center}
\end{minipage}

\end{tabular}
\end{center}
\end{figure}

 \subsection{ZMC surfaces of Riemann type foliated by hyperbolas containing a lightlike line}
By Theorem \ref{thm:causaloftimelike}, if a Type II ZMC surface of Riemann type which is foliated by circles in timelike parallel planes has exactly two causal characters, the surface contains a lightlike line (segment).
\begin{theorem} \label{thm:extimelike}
Type II ZMC surfaces of Riemann type given by (\ref{eq:type2}) which have exactly two causal characters and surfaces given by (\ref{eq:Singular_hyperbpla2}) have following characteristics of ZMC surfaces.
\begin{itemize}
\item[(1)] General type:
\begin{itemize}
\item[(i)] In the case when $a<b$, $a^2-b^2<0$, $0<c=\delta$, and $0<r^2<1/c$, the surface
\begin{equation*}
X(r,\theta ) =\left(
\displaystyle\int^r _{0} \cfrac{1}{1-cs^2}\,ds,  \int^r _{0} \cfrac{as^2}{1-cs^2}\,ds+r\sinh \theta, \int^r _{0} \cfrac{bs^2}{1-cs^2}\,ds+r\cosh\theta \right)   \nonumber
\end{equation*}
is an $\alpha ^-_{\rm II}$-type ZMC surface with the characteristic $\mu=-c^3/b^2$ along its lightlike line (see Figure 6).
\\
\item[(ii)] In the case when $a>b$, $a^2-b^2<0$, $0<c=\delta$, and $r^2, r^2_{0}>1/c$, the surface
\begin{equation*}
X(r,\theta ) =\left(
\displaystyle\int^r _{r_0} \cfrac{1}{cs^2-1}\,ds, \int^r _{r_0} \cfrac{as^2}{cs^2-1}\,ds+r\sinh \theta, \int^r _{r_0} \cfrac{bs^2}{cs^2-1}\,ds+r\cosh\theta \right)   \nonumber
\end{equation*}
is an $\alpha ^-_{\rm I}$-type ZMC surface with the characteristic $\mu=-c^3/b^2$ along its lightlike line (see Figure 7).
\end{itemize}
\item[(2)] Singular type:
every singular type surface (\ref{eq:Singular_hyperbpla2}) is an $\alpha ^-_{\rm III}$-type ZMC surface with the characteristic  $\mu=-c^3/b^2$ along its lightlike line (see Figure 8).
\end{itemize}
\end{theorem}

\begin{proof} 
First we prove (1). For any surface $X$ which satisfies the condition (i), by the proof of Theorem \ref{thm:causaloftimelike}, the lightlike line of the surface is 
\begin{equation*}
c(r) = X(r, \theta _{1})=
\cfrac{1}{c}\displaystyle\int^r _0 \cfrac{1}{\Delta (s)}\,ds\left( c, a, b \right ). \nonumber
\end{equation*}
At first, to calculate the characteristic of the surface, we find an element of the isometry group which maps the vector $(c, a, b)$ to the vector $(0, t_0, t_0 )$ for some $t_0 \in \mathbb{R} \setminus \{0\}$. By considering a rotation about the time axis, one of such elements can be written as 
\begin{center}
$  A= \left(
    \begin{array}{ccc}
    a/b &-c/b  & 0 \\
      c/b  & a/b  & 0 \\
      0 & 0 & 1 
 \end{array}
  \right) $.
  \end{center}
Therefore if we consider the immersion $\tilde{X}(r, \theta ) := A\circ X(r, \theta )$, since
\[
\Delta(s)=1-cs^2,\quad \cfrac{a}{b}\displaystyle\int^r_{0} \cfrac{1}{\Delta (s)}\,ds-\frac{c}{b}\displaystyle\int^r_{0} \cfrac{as^2}{\Delta (s)}\,ds=\frac{a}{b}r,
\] hence $\tilde{X}$ is written as
  \begin{align}
 &\tilde{X}(r, \theta) =\left( x(r, \theta),  y(r, \theta),  z(r, \theta) \right) \nonumber \nonumber \\
&=\left( \frac{a}{b}r-\frac{c}{b}r\sinh\theta , \frac{c}{b}\displaystyle\int^r_{0} \cfrac{1}{\Delta (s)}\,ds+\frac{a}{b}\displaystyle\int^r_{0} \cfrac{as^2}{\Delta (s)}\,ds+\frac{a}{b}r\sinh\theta , \displaystyle\int ^r_{0}\cfrac{bs^2}{\Delta (s)}\,ds+r\cosh\theta \right) , \nonumber
\end{align}
and the lightlike part of the surface is in $\{ (0, t, t) \mid t \in \mathbb{R} \}$. We write $\tilde{X}$ as $X$ again and determine the function $\alpha$. Noting that $\cosh \theta _{1}=b/c$, $\sinh \theta _{1}=a/c$, we can compute partial derivatives of coordinate functions on the lightlike line and get $x_{\theta}(r, \theta _{1})=-r$, $y_{\theta \theta}(r, \theta _{1})=(a^2/bc)r$, and $t_{\theta \theta}(r, \theta _{1})=(b/c)r$. By using Lemma \ref{lemma:characteristic}, we get 
\begin{equation*}
\frac{b}{c}r=\alpha (y(r))r^2+\frac{a^2}{bc}r.
\end{equation*}
Therefore we obtain $\alpha (y(r))=c/br$. \\

On the other hand, by differentiating $x(r, \theta)$ and $y(r, \theta)$ with respect to $y$,
\[
\begin{cases}{}
 0=\frac{a}{b}r_{y}-\frac{c}{b}r_{y}\sinh\theta -\frac{c}{b}r\cosh\theta \theta_{y} \\
 1=\frac{c}{b\Delta(r)}r_y + \frac{a^2r^2}{b\Delta(r)}r_y +\frac{a}{b}\sinh \theta r_y +\frac{a}{b}r\cosh \theta \theta_y , 
\end{cases}
\]
and then we get $r_{y}(r, \theta)=(c/b)\Delta(r) $, and hence,
\begin{align}
\frac{d\alpha}{dy}(y(r))+\alpha ^2(y(r))&=-\frac{c}{br^2}\times \frac{c}{b}\Delta(r) +(\frac{c}{br})^2 \nonumber \\
&=\frac{c^3}{b^2}. \nonumber
\end{align}
Therefore we obtain the characteristic $\mu=-c^3/b^2$, and hence $\alpha$ is of $\alpha^{-}$-type. Since $\alpha$ is bounded, the surface is of $\alpha^{-}_{\rm II}$-type. 

For the surface $X$ which satisfies the condition (ii), by the proof of Theorem \ref{thm:causaloftimelike}, the lightlike part of the surface is written as
\begin{equation}
c(r) =X(r, \theta _{2})=
\cfrac{1}{c}\displaystyle\int^r _{r_0} \cfrac{1}{\Delta (s)}\,ds\left( c, a, b \right ) -\cfrac{1}{c}\left( 0, a, b \right ). \nonumber
\end{equation}

Therefore if we consider the immersion $\tilde{X}(r, \theta ) := A\circ [X(r, \theta )+(1/c)\left( 0, a, b \right )]$, since 
\[
\Delta(s)=cs^2-1,\quad \frac{a}{b}\int^r_{r_0} \frac{1}{\Delta (s)}\,ds-\frac{c}{b}\int^r_{r_0} \frac{as^2}{\Delta (s)}\,ds=\frac{a}{b}r_{0}-\frac{a}{b}r,
\] and hence $\tilde{X}$ is written as
  \begin{align}
 &\tilde{X}(r, \theta) =\left( x(r, \theta),  y(r, \theta),  t(r, \theta) \right) \nonumber  \\
&=\left( -\frac{a}{b}r-\frac{c}{b}r\sinh\theta , \frac{c}{b}\displaystyle\int^r_{r_0} \cfrac{1}{\Delta (s)}\,ds+\frac{a}{b}\displaystyle\int^r_{r_0} \cfrac{as^2}{\Delta (s)}\,ds+\frac{a}{b}r\sinh\theta +\frac{a^2}{bc}r_{0}, \displaystyle\int ^r_{r_0}\cfrac{bs^2}{\Delta (s)}\,ds+r\cosh\theta +\frac{b}{c}r_{0}\right) , \nonumber
\end{align}
and the lightlike part of the surface is in $\{ (0, t, t) \mid t \in \mathbb{R} \}$. We write $\tilde{X}$ as $X$ again and determine the function $\alpha$. Noting that $\cosh \theta _{2}=-b/c$, $\sinh \theta _{2}=-a/c$, we can compute partial derivatives of coordinate functions on the lightlike line and get $x_{\theta}(r, \theta _{2})=r$, $y_{\theta \theta}(r, \theta _{2})=-(a^2/bc)r$, and $t_{\theta \theta }(r, \theta _{2})=-(b/c)r$. By using Lemma \ref{lemma:characteristic}, we get 
\begin{equation*}
-\frac{b}{c}r=\alpha (y(r))r^2-\frac{a^2}{bc}r.
\end{equation*}
Therefore we obtain $\alpha (y(r))=-c/br$. 

On the other hand, by differentiating $x(r, \theta)$ and $y(r, \theta)$ with respect to $y$,
\[
\begin{cases}{}
 0=-\frac{a}{b}r_{y}-\frac{c}{b}r_{y}\sinh\theta -\frac{c}{b}r\cosh\theta \theta_{y} \\
 1=\frac{c}{b\Delta(r)}r_y + \frac{a^2r^2}{b\Delta(r)}r_y +\frac{a}{b}\sinh \theta r_y +\frac{a}{b}r\cosh \theta \theta_y , 
\end{cases}
\]
and then we get $r_{y}(r, \theta)=(c/b)\Delta(r) $, and hence,
\begin{align}
\frac{d\alpha}{dy}(y(r))+\alpha ^2(y(r))&=\frac{c}{br^2}\times \frac{c}{b}\Delta(r) +(-\frac{c}{br})^2 \nonumber \\
&=\frac{c^3}{b^2}. \nonumber
\end{align}
Therefore we obtain the characteristic $\mu=-c^3/b^2$, and hence $\alpha$ is of $\alpha^{-}$-type. Since $\alpha$ is unbounded, the surface is of $\alpha^{-}_{\rm I}$-type.

Now we prove (2). We can check that $\alpha (y(r))=\pm c\sqrt{c}/b$ (the two signs represent the cases $a<b$ and $a>b$, respectively) and $\mu=-c^3/b^2$ by a similar calculation as in the above cases. Since the function $\alpha$ is a constant function, the surface is an $\alpha _{\rm III}^-$-type ZMC surface with the characteristic $\mu =-c^3/b^2$. The proof has been completed.
\end{proof}

 \begin{figure}
\begin{center}
\begin{tabular}{c}
\hspace{-1.8cm}
\begin{minipage}{0.45\hsize}
\begin{center}
\vspace{-1.5cm}
\hspace{-1.0cm}
\includegraphics[clip,scale=0.355,bb=0 0 300 510]{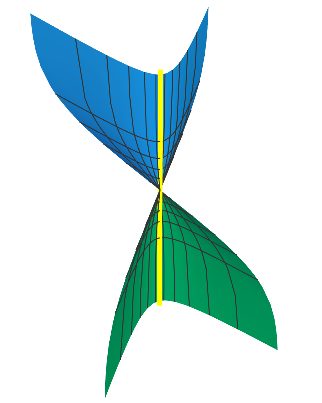}%
 \caption{An $\alpha^-_{\rm II}$-type surface foliated by hyperbolas (general type).}  
\end{center}
\end{minipage}
\hspace{-0.5cm}
\hspace{-1.0cm}
\begin{minipage}{0.45\hsize}
\begin{center}
\vspace{-2.5cm}
\includegraphics[clip,scale=0.4,bb=0 0 354 470]{Fig1b.png}%
\vspace{0.8cm}
\caption{An $\alpha^-_{\rm I}$-type surface foliated by hyperbolas (general type).}
\end{center}
\end{minipage}
\hspace{-1.0cm}
\begin{minipage}{0.45\hsize}
\begin{center}
\vspace{-0.5cm}
\includegraphics[clip,scale=0.30,bb=0 0 454 470]{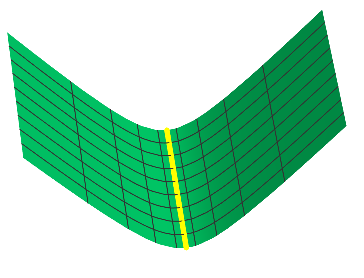}%
\vspace{0.5cm}
\caption{An $\alpha^-_{\rm III}$-type surface foliated by hyperbolas (singular type).}
\end{center}
\end{minipage}
\end{tabular}
\end{center}
\end{figure}

 \subsection{ZMC surfaces of Riemann type foliated by parabolas containing a lightlike line}
By Theorem \ref{thm:causaloflightlike}, if a ZMC surface of Riemann type which is foliated by circles in lightlike parallel planes has exactly two causal characters, the surface contains a lightlike line (segment). 
\begin{theorem} \label{thm:exlightlike}
ZMC surfaces of Riemann type given by (\ref{eq:4}) which have exactly two causal characters have the following characteristics of ZMC surfaces.
\begin{itemize}
\item[(1)] General type:
\begin{itemize}
\item[(i)] In the case when $a>0$ and $p=0$, the surface (\ref{eq:4}) is an $\alpha ^+$-type ZMC surface with the characteristic $\mu=2a^5/(a^2+b^2)^2$ along its lightlike line (see Figure 9).
\item[(ii)] In the case when $a<0$ and $p=0$, the surface (\ref{eq:4}) is an $\alpha ^-_{\rm I}$-type ZMC surface with the characteristic $\mu=2a^5/(a^2+b^2)^2$ along its lightlike line (see Figure 10).
\end{itemize}
\item[(2)] Singular type:
In the case when $a<0$ and $p=0$, the surface (\ref{eq:4}) is an $\alpha ^-_{\rm III}$-type ZMC surface with the characteristic $\mu=2a^5/(a^2+b^2)^2$ along its lightlike line (see Figure 11).
\end{itemize}
\end{theorem}
\begin{proof}
We can prove the theorem by a similar way as the proof of Theorem \ref{thm:sp}. First we prove (1). For any surface $X$ which satisfies the condition (i), by the proof of Theorem \ref{thm:causaloflightlike}, the lightlike line of the surface is
\begin{equation*}
c(u)=u \left( -\cfrac{2b}{a}, 1-\cfrac{b^2}{a^2}, -1-\cfrac{b^2}{a^2} \right )+\left( -\cfrac{2bc}{a}, -\cfrac{b^2c}{a^2}, -\cfrac{b^2c}{a^2} \right), \nonumber
\end{equation*}
where $v(u)=b/(ar(u))$.
If we take the following element of the isometry group
\begin{center}
$  A= \left(
    \begin{array}{ccc}
     (a^2-b^2)(a^2+b^2) &2ab/(a^2+b^2)  & 0 \\
     -2ab/(a^2+b^2)  & (a^2-b^2)/(a^2+b^2)  & 0 \\
      0 & 0 & -1 
 \end{array}
  \right) $
  \end{center}
and consider the immersion $\tilde{X}(u, v) := A\circ X(u, v)$, the lightlike part of the surface is in $\{ (0, t, t) \mid t \in \mathbb{R} \}$ up to an translation. We write $\tilde{X}$ as $X$ again and determine the function $\alpha$. If we write $\tilde{X}(u, v)$ as $\left( x(u, v), y(u, v), t(u, v) \right)$, by a straightforward computation, we get $x_v(u, v(u))=1$, $y_{vv}(u, v(u))=((a^2-b^2)/(a^2+b^2))r(u)$, and $t_{vv}(u, v(u))=-r(u)$. By using Lemma \ref{lemma:characteristic}, we get
\begin{equation*}
-r(u)=\alpha (y(u))+\frac{a^2-b^2}{(a^2+b^2)}r(u).
\end{equation*}
Therefore we obtain $\alpha (y(u))=-(2a^2/(a^2+b^2))r(u)$. \\
On the other hand, by differentiating $x(u, v)$ and $y(u, v)$ with respect to $y$,
\[
\begin{cases}{}
 0=\frac{a^2-b^2}{a^2+b^2}(f'(u)u_y+v_y)+\frac{2ab}{a^2+b^2}(g'(u)u_y+u_y+\frac{r'(u)u_yv^2+2r(u)vv_y}{2}) \\
 1=-\frac{2ab}{a^2+b^2}(f'(u)u_y+v_y)+\frac{a^2-b^2}{a^2+b^2}(g'(u)u_y+u_y+\frac{r'(u)u_yv^2+2r(u)vv_y}{2}) . 
\end{cases}
\]
By using the equalities $f'(u)=b/r^2(u)$, $g'(u)=b^2/2ar^2(u)$, $r'(u)=2r^2(u)+a$, and $v(u)=b/ar(u)$, we obtain $u_y(u, v(u))=a^2/(a^2+b^2)$, and hence, 
\begin{align}
\frac{d\alpha}{dy}(y(u))+\alpha ^2(y(u))&=-\frac{2a^2}{a^2+b^2}r'(u)\times \frac{a^2}{a^2+b^2} +(-\frac{2a}{a^2+b^2}r(u))^2 \\ \nonumber
&=-\frac{2a^5}{(a^2+b^2)^2}. \nonumber
\end{align}
Therefore we obtain the characteristic $\mu=2a^5/(a^2+b^2)^2>0$, and hence $\alpha$ is of $\alpha^{+}$-type.  

 For any surface $X$ which satisfies the condition (ii), the lightlike part of the surface is exactly the same as the case (i). Therefore we get $\alpha (y(u))=-(2a^2/(a^2+b^2))r(u)$ and the characteristic $\mu=2a^5/(a^2+b^2)^2<0$, and hence $\alpha$ is of $\alpha^{-}$-type.  Since $\alpha$ is non-constant and bounded, the surface is of $\alpha^{-}_{\rm I}$-type. 

Now we prove (2). We can check that
\[
\alpha (y(r))=-\sqrt{\frac{-a}{2}}\frac{2a^2}{a^2+b^2}\text{ and }\mu=\frac{2a^5}{(a^2+b^2)^2}
\]
 by a similar calculation as in the above cases. Since the function $\alpha$ is a constant function, the surface is an $\alpha _{\rm III}^-$-type ZMC surface with the characteristic $\mu=2a^5/(a^2+b^2)^2$. The proof has been completed.
\end{proof}
 \begin{figure}
\begin{center}
\begin{tabular}{c}
\hspace{-2cm}
\begin{minipage}{0.5\hsize}
\begin{center}
\vspace{-0.8cm}
\includegraphics[clip,scale=0.30,bb=0 0 504 510]{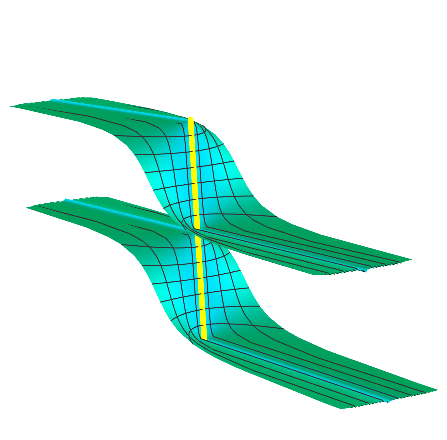}
\vspace{+0.4cm}
\caption{An $\alpha^+$-type surface foliated by parabolas and straight lines (general type).}
\end{center}
\end{minipage}
\hspace{-2cm}
\begin{minipage}{0.5\hsize}
\begin{center}
\includegraphics[clip,scale=0.30,bb=0 0 304 470]{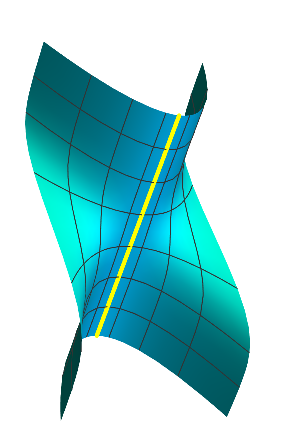}
\vspace{0.125cm}
\caption{An $\alpha^-_{\rm I}$-type surface foliated by parabolas and a straight line (general type).}
\end{center}
\end{minipage}
\hspace{-2cm}

\begin{minipage}{0.5\hsize}
\begin{center}
\includegraphics[clip,scale=0.30,bb=0 0 304 470]{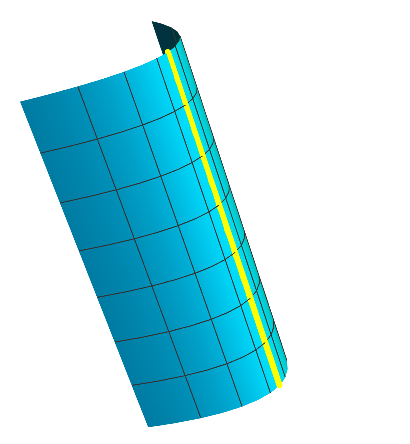}
\vspace{+0.2cm}
\caption{An $\alpha^-_{\rm III}$-type surface foliated by parabolas (singular type).}
\end{center}
\end{minipage}

\end{tabular}
\end{center}
\end{figure}
\section{New examples of ZMC entire graphs which have all causal characters}
In \cite{Kobayashi}, Kobayashi constructed the following ZMC entire graphs $t=f_i(x,y)$, $i=1,2$, which are called {\it the helicoid of the second kind} and {\it Scherk's surface of the first kind}, respectively:
 \begin{align}
 f_1(x, y):=x\tanh y,\hspace{2.6cm} \label{eq:6.1} \\
 f_2(x, y):=\log {(\cosh x)}-\log {(\cosh y)} \label{eq:6.2}.
\end{align}
In this section, as an application of Theorem \ref{thm:causaloflightlike}, we prove that there is a new ZMC entire graph of Riemann type which has all causal characters in addition to the above examples. 
\begin{theorem}\label{Therem:6}
The surfaces 
 \begin{equation}
 X(u,v) =\left(v, pe^{-2\sqrt{-2a}u}+u+\sqrt{\cfrac{-a}{2}}\cfrac{v^2}{2}, pe^{-2\sqrt{-2a}u}-u+\sqrt{\cfrac{-a}{2}}\cfrac{v^2}{2}\right),\quad (u, v)\in \mathbb{R}^2 \label{entiregraph}
 \end{equation}
 are ZMC entire graphs which have all causal characters, where $a$ and $p$ are negative constants. Moreover, these surfaces are not congruent to neither (\ref{eq:6.1}) nor (\ref{eq:6.2}). By coordinate transformations, translations, and homothetic changes, without loss of generality, we can take the constants $a=-2$ and $p=-1$. Therefore, we obtain exactly one new ZMC entire graph written as (\ref{entiregraph_intro}).
 \end{theorem}
\begin{proof}
Since the surface (\ref{entiregraph}) is a ZMC surface which belongs to the class (2) (iii) in Theorem \ref{thm:causaloflightlike}, it has all causal characters. Since $p<0$, if we define the function $\phi _{\alpha}(u)=pe^{-2\sqrt{-2a}u}+u+\alpha$ for arbitrary $\alpha$ in $\mathbb{R}$, it is bijective on $\mathbb{R}$. Using the function $\phi _{\alpha}$, the coordinate function $t$ is written as follows:
\begin{equation*}
 t=p\exp\left(-2\sqrt{-2a}(\phi _{\sqrt{\frac{-a}{2}}\frac{x^2}{2}})^{-1}(y)\right)-(\phi _{\sqrt{\frac{-a}{2}}\frac{x^2}{2}})^{-1}(y)+{\sqrt{\frac{-a}{2}}\frac{x^2}{2}}.
 \end{equation*}
Therefore, the surface (\ref{entiregraph}) is a ZMC entire graph which has all causal characters.

The surface (\ref{entiregraph}) is not congruent to neither the surface (\ref{eq:6.1}) nor (\ref{eq:6.2}). In fact, the surface (\ref{eq:6.1}) is a ruled surface (see \cite[Example 2.8]{Kobayashi}). On the other hand, by a direct calculation, we can prove that the surface (\ref{entiregraph}) is not a ruled surface (see Proposition \ref{B_Prop} in Appendix \ref{B}). Therefore the surface (\ref{entiregraph}) is not congruent to  (\ref{eq:6.1}). Moreover, the surface (\ref{eq:6.2}) is lightlike on the set $\{(x, y) \mid \tanh^2{x}+\tanh^2{y}=1\}$, that is, the lightlike part of the surface (\ref{eq:6.2}) consists of four connected component. On the other hand, the lightlike part of the surface (\ref{entiregraph}) consists of the following two disjoint curves (see Figure 12):
 \begin{align*}
c_{+}(u)  =\left(\left(\cfrac{-8pe^{-2\sqrt{-2a}u}}{\sqrt{-2a}} \right)^\frac{1}{2}, -pe^{-2\sqrt{-2a}u}+u, -pe^{-2\sqrt{-2a}u}-u\right),\quad u\in \mathbb{R},\\
c_{-}(u)  =\left(-\left(\cfrac{-8pe^{-2\sqrt{-2a}u}}{\sqrt{-2a}} \right)^\frac{1}{2}, -pe^{-2\sqrt{-2a}u}+u, -pe^{-2\sqrt{-2a}u}-u\right),\quad u\in \mathbb{R},\end{align*}
therefore the surface (\ref{entiregraph}) is not congruent to the surface (\ref{eq:6.2}).
\end{proof}
\begin{figure}[htbp]
\begin{center}
\includegraphics[clip,scale=0.40,bb=0 0 410 350]{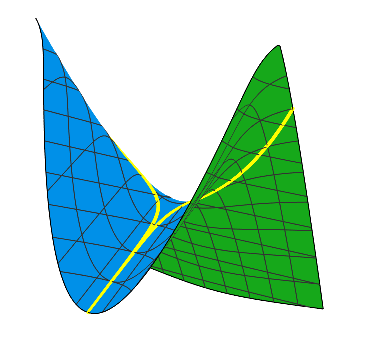}
 \vspace{0cm}
\caption{A ZMC entire graph which has all causal characters and whose lightlike curves.}
\end{center}
\end{figure}

\noindent
{\bf Acknowledgement}. The author would like to thank Professors Masaaki Umehara and Kotaro Yamada for suggesting the importance of ZMC entire graphs which have all causal characters. He is also grateful to Professor Miyuki Koiso for her many helpful comments. This work is supported by Grant-in-Aid for JSPS Fellows Number 15J06677. 

 \renewcommand{\theequation}{A.\arabic{equation}}
\setcounter{equation}{0}
\appendix
\section{The condition for a surface foliated by circles in lightlike parallel planes to be rotational}\label{A}
As pointed out in \cite{Lopez1}, surfaces which are foliated by circles in the lightlike parallel planes 
\begin{center}
$P_u=\{u(0,1,-1)+a(1,0,0)+b(0,1,1) | a,b \in \mathbb{R} \}$, $u \in J \subset \mathbb{R} $
\end{center}
are parametrized as (\ref{eq:4}) in Section \ref{Riemtype}. Each circle is the orbit of the identity component of the isometry group which fixes pointwise the axis $l=\text{span}\{ e_2+e_3\} $. In this appendix, we give a necessary and sufficient condition for a surface foliated by circles in lightlike parallel planes to be rotational.
\begin{propo}\label{A_Prop}
A surface given by (\ref{eq:4}) is rotational with the lightlike axis $l$, that is, invariant under the actions of the subgroup of the isometry group which fixes pointwise the line $l$ if and only if 
\begin{equation*}
f=0 {\text \ and\ } r(u)=-\cfrac{1}{2u}
\end{equation*}
hold.
\end{propo}
\begin{proof}
The surface given by (\ref{eq:4}) is rotational with the axis $l$ if and only if for arbitrary $\theta $ and $v\in \mathbb{R}$, there exist $\tilde{u} \in J$ and $\tilde{v} \in \mathbb{R}$ such that
\begin{center}
$ \left(
    \begin{array}{ccc}
      1 & \theta & -\theta \\
      -\theta  & 1-\theta^2/2  & \theta^2/2 \\
      -\theta & -\theta^2/2 & 1+\theta^2/2
 \end{array}
  \right) 
\left(
\begin{array}{c}
f(u)+v \\
g(u)+u +r(u)v^2/2 \\
g(u)-u +r(u)v^2/2
\end{array}
\right)
=
\left(
\begin{array}{c}
f(\tilde{u})+\tilde{v} \\
g(\tilde{u})+\tilde{u} +r(\tilde{u})\tilde{v}^2/2 \\
g(\tilde{u})-\tilde{u} +r(\tilde{u})\tilde{v}^2/2
\end{array}
\right)$,
\end{center}
that is,
\begin{center}
$\left(
\begin{array}{c}
f(u)+v+2u\theta  \\
-\theta f(u)-\theta v+g(u)+u +r(u)v^2/2-u\theta ^2 \\
-\theta f(u)-\theta v+g(u)-u +r(u)v^2/2-u\theta ^2
\end{array}
\right)=
\left(
\begin{array}{c}
f(\tilde{u})+\tilde{v} \\
g(\tilde{u})+\tilde{u} +r(\tilde{u})\tilde{v}^2/2 \\
g(\tilde{u})-\tilde{u} +r(\tilde{u})\tilde{v}^2/2
\end{array}
\right)$.
\end{center}
By the second and third equations, we get $\tilde{u}=u$. Substituting this to the first equation, we get $\tilde{v}=v+2u\theta$. By substituting $\tilde{u}$ and $\tilde{v}$ to the second equation, we obtain $f=0$ and $r(u)=-1/2u$.
\end{proof}

 \renewcommand{\theequation}{B.\arabic{equation}}
\setcounter{equation}{0}

\section{Proof of the fact that the surface (\ref{entiregraph}) is not a ruled surface}\label{B}
\begin{propo}\label{B_Prop}
The surface (\ref{entiregraph}) is not a ruled surface.
\end{propo}
 \begin{proof}
 We prove by contradiction. If the surface X given by (\ref{entiregraph}) is a ruled surface, then it has the following local representation:
 \begin{equation}
 X(u,v)=c(\tau)+sn(\tau), \label{ruled}
 \end{equation}
 where $c$ and $n$ are curves in $\mathbb{L}^3$. By differentiating (\ref{ruled}) with respect to $s$, we get
 \begin{align}\label{eq:n}
u_s\left(
\begin{array}{c}
0 \\
-2\sqrt{-2a}pe^{-2\sqrt{-2a}u}+1 \\
-2\sqrt{-2a}pe^{-2\sqrt{-2a}u}-1
\end{array}
\right)
+
v_s\left(
\begin{array}{c}
1 \\
\sqrt{-a/2}v\\
\sqrt{-a/2}v
\end{array}
\right)
=n(\tau)
=\left(
\begin{array}{c}
n_1(\tau) \\
n_2(\tau)\\
n_3(\tau)
\end{array}
\right).
\end{align}
From these equations, we obtain
\begin{subnumcases}{}
v_s(\tau ,s)=n_1(\tau),   \label{eq:v_s} \\
2u_s(\tau ,s)=n_2(\tau)-n_3(\tau). \label{eq:u_s}
\end{subnumcases}
By differentiating (\ref{eq:n}) with respect to $s$, we get
  \begin{align*}
u_s^2\left(
\begin{array}{c}
0 \\
-8ape^{-2\sqrt{-2a}u} \\
-8ape^{-2\sqrt{-2a}u}
\end{array}
\right)
+
u_{ss}\left(
\begin{array}{c}
0 \\
-2\sqrt{-2a}pe^{-2\sqrt{-2a}u}+1 \\
-2\sqrt{-2a}pe^{-2\sqrt{-2a}u}-1
\end{array}
\right)
+
v_s^2\left(
\begin{array}{c}
0 \\
\sqrt{-a/2}\\
\sqrt{-a/2}
\end{array}
\right)
+
v_{ss}\left(
\begin{array}{c}
1 \\
\sqrt{-a/2}v\\
\sqrt{-a/2}v
\end{array}
\right)
=\bm{0}.
\end{align*}
From these equations, we obtain
\begin{equation*}
v_s(\tau ,s)=\pm \left(\frac{2}{-a}\right)^{\frac{1}{4}}\sqrt{8ap}e^{-\sqrt{-2a}u(\tau ,s)}u_s(\tau ,s).   
\end{equation*}
By substituting (\ref{eq:v_s}) and (\ref{eq:u_s}) to the above equation, we get
\begin{equation}
n_1(\tau )=\pm \left(\frac{2}{-a}\right)^{\frac{1}{4}}\sqrt{2ap}e^{-\sqrt{-2a}u(\tau ,s)}(n_2(\tau )-n_3(\tau )).   \label{eq:n1 and n2-n3} 
\end{equation}
If there exists a constant $\tau _0\in \mathbb{R}$ such that $n_2(\tau _0)-n_3(\tau _0)=0$, then $n_1(\tau _0)=0$ by (\ref{eq:n1 and n2-n3}), and hence $n(\tau _0)=0$ by (\ref{eq:n}), (\ref{eq:v_s}), and (\ref{eq:u_s}). It means that the surface (\ref{ruled}) is not immersed at $\tau =\tau _0$, which is a contradiction. Therefore, we obtain $n_2(\tau )-n_3(\tau )\neq 0$ holds for arbitrary $\tau$, and by (\ref{eq:n1 and n2-n3}), 
\begin{equation*}
\frac{n_1(\tau )}{n_2(\tau )-n_3(\tau )}=\pm \left(\frac{2}{-a}\right)^{\frac{1}{4}}\sqrt{2ap}e^{-\sqrt{-2a}u(\tau ,s)}.
\end{equation*}
Since the left hand side of the above equation does not depend on $s$, we get $u_s(\tau ,s)=0$. By using (\ref{eq:u_s}), we get $n_2(\tau )-n_3(\tau )=0$, which is a contradiction, and hence we obtain the desired result.  \end{proof}
 
 \end{document}